\documentclass[a4paper,11pt]{amsart}
 \usepackage[english]{babel}
  \usepackage{amsmath,amsfonts,amssymb,amsthm,relsize,setspace,nicefrac,yhmath,amscd,eucal}
  \usepackage{amsrefs, mathrsfs}
  \usepackage{mathtools,tkz-euclide,tikz-3dplot,tkz-tab}
  \usepackage{xcolor}
  \usepackage{tikz}
  \usepackage{enumitem} 
  \usetikzlibrary{calc}
  \usepackage{comment}
  \usepackage{anyfontsize}

 \usepackage[colorlinks, linkcolor=red, citecolor=blue, urlcolor=blue, hypertexnames=true]{hyperref} 

  \usepackage{hyperref} 

 \usepackage[active]{srcltx} 
 \usepackage{color,soul} 
\newtheorem{theorem}{Theorem}
\newtheorem{corollary}[theorem]{Corollary}
\newtheorem{lemma}[theorem]{Lemma}

\usepackage[toc,page]{appendix}

\newtheorem{remark}{Remark}

\theoremstyle{definition}
\newtheorem{definition}[theorem]{Definition}

\newcommand{\be}{\begin{equation}}
\newcommand{\bel}[1]{\begin{equation}\label{#1}}
\newcommand{\ee}{\end{equation}}

\newcommand{\barr}{\begin{eqnarray}}
\newcommand{\earr}{\end{eqnarray}}
\newcommand{\bars}{\begin{eqnarray*}}
\newcommand{\ears}{\end{eqnarray*}}

\newtheorem{subn}{\name}

\newcommand{\bsn}[1]{\def\name{#1}\begin{subn}}
\newcommand{\esn}{\end{subn}}

\newtheorem{sub}{\name}[section]

\newcommand{\bs}{\begin{sub}}
\newcommand{\es}{\end{sub}}

\newcommand{\bth}[1]{\def\name{Theorem}
\begin{sub}\label{t:#1}}
\newcommand{\blemma}[1]{\def\name{Lemma}
\begin{sub}\label{l:#1}}
\newcommand{\bcor}[1]{\def\name{Corollary}
\begin{sub}\label{c:#1}}
\newcommand{\bdef}[1]{\def\name{Definition}
\begin{sub}\label{d:#1}}
\newcommand{\bprop}[1]{\def\name{Proposition}
\begin{sub}\label{p:#1}}

\newcommand{\BA}{\begin{array}}
\newcommand{\EA}{\end{array}}
\newcommand{\BAN}{\renewcommand{\arraystretch}{1.2}
\setlength{\arraycolsep}{2pt}\begin{array}}
\newcommand{\BAV}[2]{\renewcommand{\arraystretch}{#1}
\setlength{\arraycolsep}{#2}\begin{array}}
\newcommand{\BSA}{\begin{subarray}}
\newcommand{\ESA}{\end{subarray}}

\newcommand{\BAL}{\begin{aligned}}
\newcommand{\EAL}{\end{aligned}}
\newcommand{\BALG}{\begin{alignat}}
\newcommand{\EALG}{\end{alignat}}
\newcommand{\BALGN}{\begin{alignat*}}
\newcommand{\EALGN}{\end{alignat*}}
\def\angb<#1>{\langle #1 \rangle}

\newcommand{\dist}{\opname{dist}}

\def\N{\mathbb{N}}

\def\R{\mathbb{R}}

\let\.=\cdot
\let\0=\emptyset

\def\O{\Omega}

\def\dist{\text{\rm dist}}

\numberwithin{equation}{section}

\theoremstyle{definition}

\let\.=\cdot
\let\0=\emptyset

\def\O{\Omega}

\def\dist{\text{\rm dist}}

\newenvironment{formula}[1]{\begin{equation}\label{eq:#1}}
                       {\end{equation}\noindent}

\def\Fi#1{\begin{formula}{#1}}
\def\Ff{\end{formula}\noindent}

\setstcolor{red}

\usepackage{authblk}
\usepackage{changepage}

\title[]{Spectral Theory of Fractional Cooperative Systems and Threshold Dynamics in Epidemic Models}

\makeatletter
\author[1]{Cong-Bang Trang} \let\Author\@author

\affil[1]{Faculty of Fundamental Science, Industrial University of Ho Chi Minh City, 12 Nguyen Van Bao, Hanh Thong, Ho Chi Minh City, Viet Nam }

\author[$2,*$]{Hoang-Hung Vo}
\affil[$2$]{Faculty of Mathematics and Applications, Saigon University, 273 An Duong Vuong, Cho Quan, Ho Chi Minh city, Vietnam}

\begin{document}

\date{\today}


\maketitle
\begin{adjustwidth}{1.5cm}{1.5cm}
\begin{center}
\let\thefootnote\relax\footnotetext{\textit{$^1$Email address: trangcongbang@iuh.edu.vn}}

\let\thefootnote\relax\footnotetext{\textit{$^{2,*}$Corresponding author. Email address: vhhungkhtn@gmail.com}}

\Author
\end{center}
\end{adjustwidth}


\begin{abstract} The spectral analysis has long been recognized as a fundamental tool in the study of existence, uniqueness and qualitative behavior of solutions to semilinear elliptic and parabolic equations, as well as their long-time dynamics. In the contemporary mathematics, fractional Laplacian are used to model  nonlocal or long-range diffusion processes in biology, such as anomalous movement, long-distance dispersal, or Lévy-flight migration of organisms, cells, epidemic. In this paper, we adopt the definition of the spectral fractional Laplacian introduced by Caffarelli and Stinga~\cite{caffarelli_fractional_2016} to develop the eigentheory of a cooperative system modeling an infectious epidemic process and to analyze its long-term behavior. Building on Fredholm theory and the technique in Lam and Lou \cite{lam_ays_2016}, we establish a sharp criterion guaranteeing the existence and simplicity of the principal eigenvalue~$\lambda_p$, together with its variational characterizations and implications for the validity of maximum principles. In addition, the asymptotic behavior of $\lambda_p$ with respect to the diffusion coefficients, fractional orders, and domain scaling is derived, thereby complementing the results of Zhao and Ruan \cite{zhao_spatiotemporal_2023} and Feng, Li, Ruan, and Xin \cite{feng_principal_2024}. As an application of this spectral analysis, we further establish the existence, uniqueness, and threshold long-time dynamics of solutions for an endemic reaction–diffusion system with fractional diffusion, in contrast to the approach of Hsu and Yang \cite{hsu_existence_2013}. This work contributes to the modern trend of combining spectral and nonlocal analysis in applied mathematics, besides recent advances in~\cite{benguria_hadamard_2024,bisterzo_siclari_2025,ognibene_2025,zhao_spatiotemporal_2023,feng_principal_2024,feng_asymptotic_2025,nguyen_dynamics_2022}.

\end{abstract}




\tableofcontents

\section{\bf Introduction}

 In this paper, we investigate the spectral theory of the cooperative linear system with fractional Laplacian as follows
\begin{align}\label{eq:eigen}
\begin{split}
\left\{\begin{array}{lllll}
(-d_1\Delta_B)^{s_1}u+a_{11}(x)u+a_{12}(x)v=\lambda u,&x\in \O, \\
(-d_2\Delta_B)^{s_2}v+a_{21}(x)u+a_{22}(x)v=\lambda v,&x\in \O,
\end{array}\right.
\end{split}
\end{align}
where $\lambda$ denotes the eigenvalue associated with the eigenfunction ${\bf u}=(u,v)\not\equiv {\bf 0}$,  $\Omega \subset \mathbb{R}^N$ is a bounded domain with smooth boundary $\partial \Omega$, $d_1,d_2>0$ are diffusion constants and ${\bf A}=(a_{ij})_{2\times 2}$ is a matrix function. The parameters $0<s_i<1$ $(i=1,2)$ specify the fractional orders, while $(-d\Delta_B)^s$ denotes the fractional Laplacian of $-d\Delta_B$ associated with the boundary condition $B$. Here, $B=D,N$ denotes the boundary condition: Dirichlet $(B=D)$ and $(B=N)$. Following the work of Caffarelli and Stinga \cite{caffarelli_fractional_2016}, we focus on the {\it spectral fractional Laplacian}.
\begin{align*}
(-d\Delta_B)^{s}u=\sum_{k=0}^\infty d^s\mu_{B,k}^su_k\phi_k,~u\in D((-d\Delta_B)^{s}),
\end{align*}
Here, $(\mu_{B,k},\phi_k)_{k=0}^\infty$ are eigenpairs of the Laplacian $-\Delta_B$ on $\O$, 
$u_k=\langle u,\phi_k\rangle_{L^2}$ and $D((-d\Delta_B)^{s})\subset L^2(\Omega)$ 
denotes the domain of the fractional Laplacian $(-d\Delta_B)^{s}$, 
which will be discussed in more detail in Section~\ref{section:2}. 

To illustrate fractional diffusion, it is observed that human and animal movements often follow scale-free, heavy-tailed mobility patterns, which are more accurately modeled by Lévy flights (see Metzler–Klafter~\cite{metzler_random_2000}). In the whole space $\mathbb{R}^N$, such movements are represented by the fractional Laplacian $(-\Delta)^s$, defined via the principal value integral
\begin{align*}
(-\Delta)^s u(x) = C_{N,s}\,\mathrm{P.V.}\!\int_{\mathbb{R}^N}\frac{u(x)-u(y)}{|x-y|^{N+2s}}\,dy, \qquad 
C_{N,s}=\dfrac{4^s\Gamma(N/2+s)}{\pi^{N/2}|\Gamma(-s)|}.
\end{align*}
On bounded domains, the natural counterpart is the spectral fractional Laplacian, defined via eigen-expansion associated with the boundary condition. The equivalence between the integral and spectral formulations has been established in various senses by Kwaśnicki~\cite{Kwaśnicki_2017}, Caffarelli–Stinga~\cite{caffarelli_fractional_2016} and Di Nezza–Palatucci–Valdinoci~\cite{di_hitchhiker_2012}. An alternative formulation was proposed by Shieh and Spector~\cite{on_shieh_2015,on_shieh_2018}, who introduced the {\it distributional Riesz fractional gradient} $D^s$ via Riesz potentials $I_s$ and developed a functional-analytic framework for PDEs with fractional derivatives. Additional approaches, including semigroup and operator-theoretic constructions~\cite{stinga_fractional_2015,Martínez_theory_2001}, complement these perspectives and provide a unified analytic foundation.

Beyond these structural definitions, functional-analytic and entropy methods have been applied to fractional cross-diffusion systems~\cite{jungel_analysis_2022}. At the same time, spectral fractional dispersal has revealed rich dynamical phenomena, such as persistence thresholds and novel asymptotics, as shown, for instance, in Zhao and Ruan \cite{zhao_spatiotemporal_2023,zhao_spatiotemporal_2025}. Since the fractional operator is no longer of differential form, classical variational and maximum principle techniques cannot be applied directly. Consequently, many fundamental questions remain open, particularly concerning the simplicity of the principal eigenvalue, its variational characterizations and its dependence on diffusion rates and fractional orders. The focus of this work is on the properties of the principal eigenvalue $\lambda_p$ of \eqref{eq:eigen} associated with the spectral fractional Laplacian and on its applications to the analysis of endemic model dynamics.


For convenience, the eigenvalue problem \eqref{eq:eigen} can be reformulated in the more compact form: Find $\textbf{u} \in \mathcal{H}^{\textbf{s}}\setminus\{\bf 0\}$ and $\lambda\in \R$ such that
\begin{align*}
(-\textbf{d}\Delta_B)^{\textbf{s}}\textbf{u} + \textbf{Au}=\lambda \textbf{u} \text{ in the weak sense.}
\end{align*}
where
\begin{align*}
&\textbf{s}:=(s_1,s_2),~\textbf{d}:=(d_1,d_2),~\textbf{u}:=(u,v)\in \mathcal{H}^{\textbf{s}};~\textbf{0}=(0,0);\\
&(-\textbf{d}\Delta_B)^{\textbf{s}}\textbf{u}:=\left((-d_1\Delta_B)^{s_1}u,(-d_2\Delta_B)^{s_2}v\right),\quad (-\textbf{d}\Delta_B)^{\textbf{s}}:= \mathrm{diag}((-d_1\Delta_B)^{s_1},(-d_2\Delta_B)^{s_2}).
\end{align*}

We consider the following conditions for the spectral theory of \eqref{eq:eigen}.
\begin{enumerate}[label=(A\arabic*)]  

\item \label{cond:cond11} $\textbf{A}=(a_{ij})=\left(\begin{matrix}
a_{11}&a_{12}\\
a_{21}&a_{22}
\end{matrix}\right)\in [C^{0,\alpha}(\overline{\O})\cap C^{1,0}(\overline{\O})]^{2\times 2}$ for some $0<\alpha<1$, $\dfrac{1}{2}<\alpha-s_i$, $\alpha+s_i<1$ for $i=1,2$.

\item  \label{cond:cond21} Suppose that $a_{12},~a_{21}<0$ and $\partial_{\textbf{n}} a_{ij}|_{\partial \O}=0$ for any $i,j=1,2$.

\end{enumerate}
Here, $\partial_{\textbf{n}} := \dfrac{\partial}{\partial \textbf{n}}$ 
denotes the derivative in the direction of the outward unit normal vector. The Hölder order $0<\alpha<1$ is chosen so that when the regularity arguments are applied, the functional space in our work do not fall into the following "irregular" spaces $\mathcal{C}(\overline{\O}),~\mathcal{C}^1(\overline{\O}),~\mathcal{C}^1_B(\overline{\O})$ in Zhao and Ruan \cite[Section 2.2]{zhao_spatiotemporal_2025} and Lunardi \cite[Theorem 3.1.30]{lunardi_1995}. On the other hand, since the matrix function $\mathbf{A}_0 = -\mathbf{A}$ is cooperative, the Perron–Frobenius theorem (see \cite{gantmacher_theory_1959}) guarantees the existence of a principal eigenvalue $\overline{\lambda}(\mathbf{A}_0)$, characterized by $\overline{\lambda}(\mathbf{A}_0)>\operatorname{Re}(\lambda')$ for any other eigenvalue $\lambda'$ of $\mathbf{A}_0$. Consequently, the principal eigenvalue of $\mathbf{A}$, still denoted by $\overline{\lambda}(\mathbf{A})$, can be obtained, which enjoys analogous properties except that $\overline{\lambda}(\mathbf{A}) < \operatorname{Re}(\lambda)$ for any other eigenvalue $\lambda$ of $\mathbf{A}$. Furthermore, based on Lam-Lou \cite[Claim~4.2]{lam_ays_2016}, if $x \mapsto \mathbf{A}(x)$ is continuous, then the eigenpairs $x \mapsto (\lambda_i(x), \phi_i(x))$ of $\mathbf{A}(x)$ can be chosen continuously for $i=1,\dots,N$, with $|\phi_i(x)|_N = 1$. The same conclusion holds for analyticity (see Kato \cite[Chapter~2]{kato_1976} and Kriegl and Michor \cite{kriegl_michor_2003}).

\medskip


For decades, the spectral theory of elliptic systems is a central theme in partial differential equations, as the principal eigenvalue determines positivity properties, maximum principles, bifurcation thresholds and the dynamics of nonlinear parabolic models. In the classical \textit{local} setting, the pioneering work of Berestycki, Nirenberg, Varadhan \cite{berestycki_nirenberg_varadhan_1994} developed the theory of generalized principal eigenvalues for general second-order elliptic operators in bounded domains. Their framework established existence, simplicity and asymptotic behavior of the generalized principal eigenvalue, with important applications to biological dynamics. For cooperative systems, the Krein–Rutman theorem ensures the existence of a positive principal eigenfunction and bifurcation theory has been applied to describe steady states \cite{thieme_spectral_2009,wang_basic_2012}. 

In an elegant work, Lam and Lou \cite{lam_ays_2016} analyzed the existence and asymptotic behavior of the principal eigenvalue of general linear cooperative elliptic systems with small diffusion rates.
\begin{align*}
\begin{split}
\left\{\begin{array}{lllll}
D\mathcal{L}\phi+A\phi+\lambda \phi =0, &\text{ in }\O,\\
\mathcal{B}\phi = 0, &\text{ on }\partial \O,
\end{array}\right.
\end{split}
\end{align*}
where $\O\subset \R^N$ is a bounded domain with smooth boundary, $D=\text{diag}(d_1,...,d_n)$, $d_i>0$ is diffusion constant,
$A=(a_{ij})\in (C(\overline{\Omega}))^{n\times n}$ satisfies $a_{ij}(x)\geq 0$ in $\Omega$ when $i\neq j$, $\mathcal{L}=\text{diag}(L_1,...,L_n)$ with $L_i$ being second-order elliptic operators of non-divergence form, i.e. for any $1\leq i \leq n$ and $1\leq k,l\leq N$, 
\begin{align*}
L_iu:=\alpha^i_{kl}\partial^2_{x_kx_l}u+\beta^i_k\partial_{x_k}u+\gamma^iu,
\end{align*}
Here, $\alpha^i_{kl},\;\beta^i_k,\;\gamma^i\in C^\theta(\overline{\Omega})$ for some $\theta\in(0,1)$ and $\eta_0|\xi|^2<\alpha^i_{kl}(x)\,\xi_k\xi_l<\eta_1|\xi|^2,~\forall \xi\in\mathbb{R}^N,\ x\in\Omega$ and for some positive constants $\eta_0,\eta_1$, $\phi=(\phi_1,\dots,\phi_n)^T\in [C^2(\overline{\Omega})]^n$  and 
$\mathcal{B}=(B_1,\dots,B_n)$ are boundary operators satisfying for each $i$ either the Robin boundary condition or the Dirichlet boundary condition.
The authors first proved the existence of the principal eigenvalue by combining the Krein–Rutman theorem with the continuity of the spectral radius of a suitable operator. Using additional comparison results, they further investigated the asymptotic behavior of the principal eigenvalue as $\max\limits_{i\in\{1,\dots,n\}} d_i \to 0$. In the final part of their work, they further analyzed the long-time dynamics of a reaction–diffusion system by applying the spectral theory of cooperative systems. Building on this interest, the asymptotic behavior of eigenvalues with respect to boundary parameters has attracted considerable attention in the community. Recently, Bisterzo and Siclari~\cite{bisterzo_siclari_2025} studied quantitative spectral stability for operators with compact resolvent on $L^2(X,m)$, under general assumptions and characterized the dominant term in the asymptotic expansion of eigenvalue variations, where $(X,m)$ is a measure space. Their framework unifies and extends several known results, with applications ranging from the Neumann limit of Robin problems and conformal transformations of Riemannian metrics to Dirichlet forms with small-capacity perturbations and families of Fourier multipliers. Alternatively, Ognibene~\cite{ognibene_2025} studied the asymptotic properties of Robin eigenvalues as the boundary parameter diverges, focusing on the rate at which they converge to their Dirichlet counterparts. For further studies of principal eigenvalues in the local setting, we refer to \cite{bai_he_2020,anton_lopez_2019,lam_Introduction_2022,Lemenant_2013}.

Motivated by the connection between spectral theory and population dynamics, significant attention has been devoted to epidemic models, where eigenvalue-based thresholds determine whether a disease persists or dies out. Such models have a long history, particularly in the study of man–environment–man diseases. A classical example is the work of Capasso and Paveri-Fontana \cite{capasso_paveri_1979}, who introduced a model for the cholera epidemic that struck the Mediterranean regions in 1973. The model is given by
\begin{align*}
\left\{\begin{array}{llll}
u_t=-au+\alpha v,\\
v_t=-bv+G(u),
\end{array}\right. t>0
\end{align*}
Here, $u,~v$ represent the concentration of bacterial population  and the infected human population in the environment, respectively, $-au,~-bv$ are the nature diminishing rate of the bacterial population and the infectious population, respectively and $\alpha v$ is the contribution of the infective humans to the growth rate of bacteria. Furthermore, the last term $G(u)$ represents the \textit{force of infection} on the human population under the assumption that the number of susceptibles remains constant during the epidemic. As noted in \cite{cantrell_schmitt_1986}, $G$ grows linearly with the concentration of the infectious agent when $u$ is small. On the other hand, for large $u$, this linearity becomes biologically unrealistic and $G$ is typically modeled as a nonlinear, strictly increasing and concave function. Subsequent works of Capasso and coauthors \cite{capasso_maddalena_1981,capasso_wilson_1997} introduced partially degenerate reaction–diffusion models for fecal–oral transmission in Mediterranean regions. These frameworks have attracted sustained attention in the theory of reaction–diffusion equations, with contributions in both the pure and applied mathematics. For example, to investigate the endemic model via cooperative system, Hsu and Yang \cite{hsu_existence_2013} proposed structural conditions on nonlinearities, which has played an important role in the travelling wave literature, as the authors ensure biologically realistic infection terms that saturate at high densities. The system is stated as follows
\begin{align*}
\left\{\begin{array}{llll}
u_t = d_1 u_{xx} - \alpha_1 u + H(v), \\
v_t = d_2 v_{xx} - \alpha_2 v + G(u),
\end{array}\right.~x\in \R,~t>0,
\end{align*}
 where $u$ denotes the density of a pathogen in the environment and $v$ the density of infected individuals. The nonlinearities $H$ and $G$ were assumed to be increasing, concave and saturating, reflecting the fact that infection rates rise with exposure but eventually level off due to natural limitations such as resource constraints or host immunity. A key feature of their analysis is the characterization of the endemic equilibrium $(K_1, K_2)$, defined as the unique positive solution of the algebraic system  
\begin{align*}
\left\{\begin{array}{llll}
-\alpha_1 K_1 + H(K_2) = 0, \\[6pt]
-\alpha_2 K_2 + G(K_1) = 0.
\end{array}\right.
\end{align*}
This equilibrium represents the coexistence of pathogen and infected hosts and provides the asymptotic state approached by travelling wave solutions. In this direction, Wu and Hsu \cite{wu_hsu_2016} showed the existence of entire solutions for delayed monostable epidemic models via the traveling wave existence. For further endemic investigation in local setting, we refer the readers to \cite{allen_asymptotic_2008,ge_sis_2015}. Very recently, in a rigorous analysis, Nguyen and Vo~\cite{nguyen_dynamics_2022} analyzed a free boundary system with coupled nonlocal diffusions modeling digestive-tract epidemics such as cholera. The authors established well-posedness of the problem, derived threshold conditions for persistence versus extinction in terms of basic reproduction numbers and obtained sharp spreading–vanishing criteria. Their approach relied on variational characterizations of principal eigenvalues together with the nonlocal maximum principle and sliding methods. As a complementary contribution, Ninh and Vo~\cite{ninh_vo_asymptotic_2025} investigated nonlocal cooperative systems with spatial heterogeneity, where compactness arguments such as the Krein–Rutman theorem cannot be directly applied. By applying the Lax–Milgram theorem, the authors proved the existence and simplicity of the principal eigenvalue, provided a counterexample to nonexistence and described its asymptotic dependence on dispersal rate and dispersal range. Together, these works highlight both the spectral foundation and the dynamical consequences of principal eigenvalues in nonlocal epidemic models. From another perspective, Benguria, Pereira and Sáez~\cite{benguria_hadamard_2024} carried out a deep study of Hadamard-type formulae for both simple and multiple eigenvalues in a class of nonlocal eigenvalue problems. Their analysis includes, among others, the classical nonlocal problems with Dirichlet and Neumann boundary conditions. In this framework, the Hadamard formula is derived under domain perturbations arising from embeddings of $n$-dimensional Riemannian manifolds (possibly with boundary) of finite volume. For additional investigations of the principal eigenvalue in the nonlocal setting, we refer the reader to \cite{vo_timeperiodic_neumann,shen_vo_2019,coville_simple_2010,bao_shen_2017,feng_principal_2024,wu_principal_2025}. Since then, numerous advances inspired by spectral theory have been made, addressing both theoretical questions and applied problems. Substantial progress has been achieved on long-time dynamics, wave propagation and spreading speed estimates~\cite{weinberger_1982,weinberger_lewis_li_2002,xu_spatial_2021,vo_ta_vudo_propagation}. Moreover, the role of principal eigenvalues in age-structured models with nonlocal dispersal has been systematically investigated in a series of works by Kang and Ruan~\cite{kang_principal_2022,kang_principal_2023} and by Ducrot-Kang-Ruan~\cite{ducrot_age_2024,ducrot_age-structured_2024}. These contributions demonstrate how spectral theory provides a unifying framework across local, nonlocal and age-structured population models.

A natural further step is to consider fractional diffusion, where the spectral theory presents additional challenges, particularly in regularity and positivity, due to the nonlocal nature of the operator and the absence of the strong maximum principle for the spectral fractional Laplacian. To address this, Zhao and Ruan~\cite{zhao_spatiotemporal_2025} carried out a rigorous analysis to deepen the understanding of the spectral theory for the fractional Laplacian in both elliptic and time-periodic settings. More precisely, they studied the following eigenvalue problems.
\begin{align*}
(-dL_B)^s \psi + \mu c(x)\psi = \lambda \psi, 
&~ x\in \Omega,
\end{align*}
and
\begin{align*}
\left\{
\begin{array}{ll}
\psi_t + (-dL_B)^s \psi + c(t,x)\psi = \lambda \psi, 
& (t,x)\in \mathbb{R}\times \Omega, \\
\psi(t+T,x) = \psi(t,x), 
& (t,x)\in \mathbb{R}\times \overline{\Omega},
\end{array}
\right.
\end{align*}
where $(\lambda,\psi)$ is an eigenpair, $L_B$ is a classical second-order linear elliptic operator on the bounded domain $\O$ and $B$ specifies oblique derivative boundary conditions, including the Dirichlet and Neumann cases. The authors employed the Balakrishnan–Komatsu definition to investigate the general fractional Laplacian $(-dL_B)^s$, which in the simplest case reduces to the spectral fractional Laplacian $(-d\Delta_B)^s$. In this work, several key properties were further developed, including the regularity of eigenfunctions and the strong maximum principle in both elliptic and time-periodic settings. For additional properties of the Balakrishnan–Komatsu definition, we refer readers to \cite{kato_1976,pazy_1983,ducrot_magal_prevost_2010}. Following this approach, in a celebrated work, Zhao and Ruan \cite{zhao_spatiotemporal_2023} proposed an SIS-type epidemic model involving the spectral fractional Laplacian, subject to Neumann boundary conditions and incorporating logistic source terms.
\begin{align*}
\left\{\begin{array}{llll}
u_t + ( - d_u\Delta_N)^{s_1} u = a(x)u - b(x)u^2 - \dfrac{p(x) u v}{u+v} + q(x)v,\\[6pt]
v_t + ( - d_v\Delta_N)^{s_2} v = \dfrac{p(x) u v}{u+v} - q(x)v,
\end{array}\right.~t>0,~x\in \O,
\end{align*}
where $u$ and $v$ denote the densities of the susceptible and infected populations. Here, $0<s_i<1$ $(i=1,2)$ are the fractional powers of the Neumann Laplace operator and $(-d\Delta_N)^s$ represents the Neumann fractional Laplacian with diffusion constant $d=d_u,d_v$. The authors focused on the spatio-temporal dynamics of the system in this work. To this end, they first proved the existence of the disease-free equilibrium (DFE) and the endemic equilibrium (EE). Then, they established the uniqueness and stability of (DFE) and analyzed the asymptotic behavior of the (EE) under parameter regimes such as $d_u \to 0$, $d_u \to \infty$ and $d_v \to \infty$. Their analysis relied on a careful use of the spectral theory of the fractional Laplacian. 

Alternatively, from a variational perspective, Bucur, Dipierro, Lombardini, Maz\'on and Valdinoci~\cite{bucur_dipierro_lombardini_mazon_valdinoci_2023} analyzed the convergence of \((s,p)\)-energies as \(p \downarrow 1\). The authors established both pointwise and \(\Gamma\)-convergence to the \(W^{s,1}\)-energy, proved convergence of the associated Euler--Lagrange equations and addressed regularity of minimizers, thereby linking nonlocal energies with their local counterparts. Together, these works advance the understanding of \textit{local--nonlocal interactions}, highlighting their role in persistence thresholds of populations, multiplicity phenomena in nonlinear elliptic problems and variational limits in fractional frameworks. On the other hand, Dipierro, Proietti-Lippi and Valdinoci \cite{dipierro_proietti_valdinoci_2024} examined diffusive populations subject to the \textit{Allee effect}, where persistence requires overcoming a critical density threshold.
\begin{align*}
\left\{
\begin{array}{ll}
\partial_t u(x,t) - \alpha \Delta u(x,t) + \beta (-\Delta)^s u(x,t) 
= (u(x,t)-a)(\rho - u(x,t))u(x,t), 
& \text{in } \Omega \times (0,+\infty), \\[1ex]
u(x,0) = u_0(x), 
& \text{in } \Omega, \\[1ex]
\text{with } (\alpha,\beta)\text{--Neumann conditions}, 
& \text{in } (0,+\infty),
\end{array}
\right.
\end{align*}
Here, $\O$ is a bounded domain with $C^1$-boundary, $a,~\rho\in \R$ are constants and $(\alpha,\beta)\text{--Neumann conditions}$ is defined by
\begin{align*}
\left\{
\begin{array}{ll}
\mathcal{N}_s u(x) = 0 \quad \text{for all } x \in \mathbb{R}^n \setminus \Omega, 
& \text{when } \alpha = 0, \\[1ex]
\partial_{\bf n} u(x) = 0 \quad \text{for all } x \in \partial \Omega, 
& \text{when } \beta = 0, \\[1ex]
\mathcal{N}_s u(x) = 0 \quad \text{for all } x \in \mathbb{R}^n \setminus \overline{\Omega}, \quad 
\partial_{\bf n} u(x) = 0 \quad \text{for all } x \in \partial \Omega, 
& \text{when } \alpha \neq 0 \ \text{and } \beta \neq 0.
\end{array}
\right.
\end{align*}
where $\mathcal{N}_s$ denotes the generalized Neumann boundary condition for the fractional Laplacian. By incorporating dispersal through Brownian motion, L\'evy flights and mixed random strategies observed in nature, the authors established conditions for the existence and nonexistence of stationary states, investigated the associated evolution problem via energy methods and monotonicity arguments and derived long-time behavior, including the case of an ``inverse'' Allee effect. As a complementary contribution, Su, Valdinoci, Wei and Zhang~\cite{su_valdinoci_wei_zhang_2024} investigated a \emph{mixed local–nonlocal semilinear elliptic equation} driven by Brownian and Lévy processes, of the form
\begin{equation*}
\left\{
\begin{array}{ll}
-\Delta u + (-\Delta)^s u = \lambda |u|^{p-2}u + g(x,u), & \text{in } \Omega, \\[1ex]
u = 0, & \text{in } \mathbb{R}^n \setminus \Omega,
\end{array}
\right.
\end{equation*}
Under broad conditions, the study establishes multiplicity results by first proving the existence of at least five weak solutions in bounded domains. Beyond this, refined variational techniques—namely the descending flow method and the Nehari manifold framework—yield a richer solution structure: at least six distinct classical solutions are obtained, comprising positive, negative and sign-changing states, each with carefully characterized energy levels. We refer the readers to 
\cite{Spatiotemporal_li_2024,dipierro_superposition_neumann_2025,mugnai_pinamonti_vecchi_brezis_oswald_2020,dob_fractional_system_resonance_2022,saadi_fractional_distributional_2021} for further investigations on fractional Laplacian models and to \cite{cowan_smaily_feulefack_mixed_2025,dipierro_neumann_superposition_2025} for additional studies on the principal eigenvalue.



 


Returning to system \eqref{eq:eigen}, the objective of this paper is to investigate the existence, simplicity and positivity of the eigenfunction via the Krein–Rutman theorem and Fredholm theory, as well as the asymptotic behavior of the principal eigenvalue $\lambda_p$, in dependence on the diffusion constants $d_1,d_2$, the fractional orders $s_1,s_2$ and the expansion–contraction of the domain $\Omega$ under conditions \ref{cond:cond11} and \ref{cond:cond21}.


For clarity of presentation, we introduce the following notation, which will be used throughout the paper.
\begin{align*}
&\mathcal{X}:=L^2(\O)\times L^2(\O);~ \mathcal{H}^{\textbf{s}}:=H^{s_1}{(\O)}\times H^{s_2}{(\O)};~\mathcal{C}:=C(\overline{\O})\times C(\overline{\O});\\
&\underline{s}=\min\{s_1,s_2\};~p^*:=p^*(N,s)=\dfrac{Np}{N-sp}.
\end{align*}
The norms for $\mathcal{X},~\mathcal{C}$ and $\mathcal{H}^{\textbf{s}}$ are the usual product norms. For convenience, we further define
\begin{align*}
\mathcal{K}_B:=(-\textbf{d}\Delta_B)^{\textbf{s}} + {\bf A}.
\end{align*}

In the original work, Berestycki-Nirenberg-Varadhan \cite{berestycki_nirenberg_varadhan_1994} defined the generalized principal eigenvalue for the general elliptic operator $L$ in the form
\begin{align*}
L=a_{ij}(x)\partial_{ij}+b_i(x)\partial_i+c(x),
\end{align*}
which, in the simplest case, reduces to the Laplacian $\Delta$. As the results, defining the generalized principal eigenvalue for the fractional Laplacian $(-\Delta)^s$ is rather challenging. Motivated by the works of Nguyen–Vo \cite{nguyen_dynamics_2022} and Ninh–Vo \cite{ninh_vo_asymptotic_2025}, we introduce $\lambda_p'$ and $\lambda_p''$ as the \emph{generalized principal eigenvalues}, defined by
\begin{align*}
&\lambda_p'\left(\mathcal{K}_B\right) := \inf\{\lambda\in \R:\exists \boldsymbol{\phi}\in \mathcal{C}\cap \mathcal{H}^{\bf s},~\boldsymbol{\phi}>{\bf 0},~-\mathcal{K}_B{\bf \boldsymbol{\phi}}+\lambda\boldsymbol{\phi} \geq  {\bf 0} \},\\
&\lambda_p''\left(\mathcal{K}_B\right) := \sup\{\lambda\in \R:\exists \boldsymbol{\phi}\in \mathcal{C}\cap \mathcal{H}^{\bf s},~\boldsymbol{\phi}>{\bf 0},~-\mathcal{K}_B {\bf \boldsymbol{\phi}}+\lambda\boldsymbol{\phi} \leq  {\bf 0} \}.
\end{align*}

We study the existence of the principal eigenvalue $\lambda_p\left(\mathcal{K}_B\right)$ for $B=D,N$ as follows

\begin{theorem}\label{eigen_theo}
Assume \ref{cond:cond11} and \ref{cond:cond21} hold. Then, \eqref{eq:eigen} admits a simple, principal eigenvalue $\lambda_p\left(\mathcal{K}_B\right)$. Furthermore, suppose that $a_{12}=a_{21}$, one has
\begin{align*}
\lambda_p\left(\mathcal{K}_B\right)=\inf\left\{ \mathcal{J}(\textbf{u}):\textbf{u}\in \mathcal{D}(\mathcal{J}),~\|\textbf{u}\|_{\mathcal{X}}=1\right\},
\end{align*}
and
\begin{align*}
\lambda_p\left(\mathcal{K}_B\right)=\lambda_p'(\mathcal{K}_B)=\lambda_p''(\mathcal{K}_B),
\end{align*}
where $\mathcal{D}(\mathcal{J})=\mathcal{H}^{\textbf{s}}$ if $B=N$ and $\mathcal{D}(\mathcal{J})=\mathbb{H}^{s_1}\times\mathbb{H}^{s_2}$ if $B=D$ and
\begin{align*}
\mathcal{J}(\textbf{u}):=\left<\mathcal{K}_B{\bf u},{\bf u}\right> =\left<(-{\bf d}\Delta_B)^{\textbf{s}}{\bf u},{\bf u}\right> + \left<{\bf Au},{\bf u}\right>,~{\bf u}\in \mathcal{D}(\mathcal{J}).
\end{align*}

\end{theorem}

Thanks to the compactness of $((-\textbf{d}\Delta_B)^{\textbf{s}} + \beta I)^{-1}$ on $\mathcal{C}$ for sufficiently large $\beta > 0$, we apply Krein--Rutman theorem to establish the existence and simplicity of the principal eigenvalue $\lambda_p(\mathcal{K}_B)$ on $\mathcal{C}$. The further regularity of the principal eigenfunction $\varphi_1$ is obtained by applying \cite[Theorem~4.1]{grubb_regularity_2016} and \cite[Corollary~2.2]{zhao_spatiotemporal_2025}. Moreover, under the additional condition $a_{12} = a_{21}$, the Lax–Milgram theorem and the Fredholm alternative are applied to characterize the spectrum of $\mathcal{K}_B$, which in this case is self-adjoint on $\mathcal{X}$.


For future purpose, we define 
\begin{align*}
\lambda_p:=\lambda_p\left(\mathcal{K}_N\right);~\lambda_p^D:=\lambda_p\left(\mathcal{K}_D\right).
\end{align*}

Now, let us discuss the non-triviality of the principal eigenfunction $\varphi_1>0$ associated with the principal eigenvalue $\lambda_p$. It is easy to see that
\begin{align*}
\text{ $\varphi_1\equiv{\bf C}$ for some constant vector ${\bf C}=(C_1,C_2)>{\bf 0}$  if and only if ${\bf A}(x){\bf C}=\lambda_p {\bf C}$ for any $x\in \overline{\O}$. }
\end{align*}
Hence, under the non-trivial condition for the matrix function ${\bf A}(x)$, 
i.e., 
\begin{center}
${\bf A}(x){\bf C}=\lambda_p{\bf C}$ for any $x\in \overline{\O}$ implies ${\bf C}\equiv 0$, 
\end{center}
the possible cases for the eigenfunction $\varphi_1$ are 
\[
\varphi_1=(\varphi_{1,1},\varphi_{1,2}), \quad 
\varphi_1=(C_1,\varphi_{1,2}) \quad \text{or} \quad 
\varphi_1=(\varphi_{1,1},C_2),
\]
for some constants $C_1,~C_2>0$. Based on these reasons, we additionally assume the non-trivial condition for the eigenfunction $\varphi_1$.
\begin{enumerate}[label=(A\arabic*)]  
\setcounter{enumi}{2}
\item \label{cond:cond31} $\varphi_1=(\varphi_{1,1},\varphi_{1,2})$ is not identically equal to some constant vector ${\bf C}>{\bf 0}$. There exists $i=1,2$ such that $\varphi_{1,i}$ is non-constant.
\end{enumerate}

Next, the effects of the diffusion rate ${\bf d}=(d_1,d_2)$ are investigated as follows
\begin{theorem}\label{deffect}
Assume \ref{cond:cond11} and \ref{cond:cond21} are satisfied. Suppose further that $a_{12}=a_{21}$. Then, the following statements are valid.

\begin{enumerate}[label=(\roman*)]
\item \label{monotone} $\lambda_p({\bf d})$ is increasing and concave. Furthermore, it is analytic on $\textbf{d}\in (0,\infty)\times (0,\infty)$. In addition, if \ref{cond:cond31} holds for some $i=1,2$ for any ${\bf d}$, with $d_i<d_i'$, one can check that
\begin{align*}
\left\{\begin{array}{llll}
\lambda_p(d_1,d_2)<\lambda_p(d_1',d_2),& \text{ if }i=1,~\forall d_2>0,\\
\lambda_p(d_1,d_2)<\lambda_p(d_1,d_2'),&\text{ if }i=2,~\forall d_1>0.
\end{array}\right.
\end{align*}

\item \label{derivative} The gradient of $\lambda_p ({\bf d})$ is as follows
\begin{align*}
\begin{array}{lllll}
\partial_{d_i}\lambda_p({\bf d})=\dfrac{1}{\|\varphi _{1}^{\bf d}\|_{\mathcal{X}}^2}s_i d_i^{-1} \left<(-d_i\Delta_N)^{s_i}\varphi _{1,i}^{\bf d},\varphi_{1,i}^{\bf d}\right>\geq 0,~i=1,2,
\end{array}
\end{align*}
 where $\varphi_1^{\bf d}$ is the appropriate positive eigenfunction associated with $\lambda_p({\bf d})$ and $\partial_{d_i}\lambda_p({\bf d})$ is the partial derivative of $\lambda_p({\bf d})$ with respect to $d_i$, $i=1,2$. Suppose further that \ref{cond:cond31} holds for some $i=1,2$ and ${\bf d}^0=(d_1^0,d_2^0)$, then one can check that
 \begin{align*}
\partial_{d_i}\lambda_p({\bf d}^0)>0.
 \end{align*}

\item \label{lim0} The zero-limits are as follows
\begin{align*}
\lim\limits_{\max\{d_1,d_2\}\rightarrow 0} \lambda_p({\bf d})=\lim\limits_{\max\{d_1,d_2\}\rightarrow 0} \lambda_p^D({\bf d})=\min\limits_{x\in \overline{\O}}\overline{\lambda}({\bf A}(x)).
\end{align*}

\item \label{liminf} The infinity-limits are as follows
\begin{align*}
\lim\limits_{\min\{d_1,d_2\}\rightarrow \infty} \lambda_p({\bf d})=\overline{\lambda}({\bf \overline{A}});\quad\lim\limits_{\min\{d_1,d_2\}\rightarrow \infty} \lambda_p^D({\bf d})=\infty.
\end{align*}
where $\overline{\bf A}:=\dfrac{1}{|\O|}\displaystyle\int_{\O}{\bf A}(x)dx$ and  $\overline{\lambda}(\overline{\bf A})$ is the principal eigenvalue of matrix $\overline{\bf A}$.

\end{enumerate}

\end{theorem}

The technique is to apply the holomorphic branch of $\left\{\varphi^{\bf d}_1\right\}_{{\bf d}>{\bf 0}}$ from the work of Kato \cite{kato_1976} to investigate the differentiability of $\lambda_p({\bf d})$. To analyze concavity, it is well-known that if $f_i$ is concave, then so is $\inf_{i} f_i$. This implies the concavity of the principal eigenvalue. Regarding the limits of the principal eigenvalues $\lambda_p$ and $\lambda_p^D$ as $\max\{d_1,d_2\}\to 0$, several auxiliary lemmas are established, including a comparison of $\lambda_p^D$ between distinct domains. For the case $\min\{d_1,d_2\}\to \infty$, in the Neumann setting, the eigenfunction is shown to converge asymptotically to a constant vector, which coincides with the eigenvector of $\overline{\mathbf A}$. The limiting behavior then follows from an application of Fatou’s lemma. In contrast, in the Dirichlet case, it can be directly verified that the lower bound is determined by the principal eigenvalue of $(-d_i\Delta_D)^{s_i}$ for $i=1,2$, thereby yielding the desired conclusions.

Next, we study the effects of the fractional order ${\bf s}=(s_1,s_2)$.

\begin{theorem}\label{fracorder1}
Assume \ref{cond:cond11} and \ref{cond:cond21} hold. Suppose further that $a_{12}=a_{21}$. Then, the following statements are valid.

\begin{enumerate}[label=(\roman*)]
\item \label{analytic} $\lambda_p({\bf s})$  is analytic on ${\bf s}\in (0,1)\times (0,1)$.

\item \label{limit1s1} The limit for the fractional order ${\bf s}$ is given as follows.
\begin{align*}
\lim\limits_{\min\{s_1,s_2\}\rightarrow 1^-}\lambda_p({\bf s})=\lambda_1\left(-{\bf d}\Delta_N+ {\bf A}\right),
\end{align*}
where $\lambda_1\left(-\textbf{d}\Delta_N+ {\bf A}\right)$ is the principal eigenvalue of $-\textbf{d}\Delta_N+ {\bf A}$.

\item \label{limit1s0} The limit for the fractional order ${\bf s}$ is given as follows.
\begin{align*}
\lim\limits_{\max\{s_1,s_2\}\rightarrow 0^+}\lambda_p({\bf s})=\lambda_1({\bf A}+ I-{\bf P}_0),
\end{align*}
where ${\bf P}_0:=(P_0,P_0)$, $P_0u=|\O|^{-1}\displaystyle\int_{\O}udx,~u\in L^2(\O)$, $\lambda_1({\bf A}+ I-P_0)$ is the principal eigenvalue of ${\bf A}+ I-{\bf P}_0$.

%
\end{enumerate}
\end{theorem}

The key to establishing this result lies in the asymptotic behavior of the fractional Laplacian 
$(-\Delta)^s$ as $s \to 1^-$ and $s \to 0^+$, together with the compact embedding of 
$H^s(\Omega)$ into $L^2$, which guarantees the convergence of the eigenfunction. Together with Fatou's lemma, we obtain the behavior of $\lambda_p({\bf s})$

We emphasize that the main difficulty in Theorems \ref{deffect} and \ref{fracorder1} 
arises from the regularity of the eigenfunction, which is delicate due to the structure of the fractional Sobolev space $H^s(\Omega)$ (see Lions and Magenes \cite[Chapter~1]{lions_pro_1968} 
and Zhao and Ruan \cite[Section~2]{zhao_spatiotemporal_2025} for more details).

\phantom{1}

As an application, we employ the spectral theory to study the following endemic model.
\begin{align}\label{eq:main}
\begin{split}
\left\{\begin{array}{lllll}
u_t +(-d_1\Delta_N)^{s_1}u=-a(x)u+H(v),&t>0,~x\in \O, \\
v_t+(-d_2\Delta_N)^{s_2}v=-b(x)v+G(u),&t>0,~x\in \O, \\
u(0,x) = \ u_0(x),\ 
v(0,x) = \ v_0(x),
\end{array}\right.
\end{split}
\end{align}
where $u=u(t,x)$ and $v=v(t,x)$ denote the spatial densities of the bacterial population 
and the infective human population, respectively, at time $t>0$ and position $x \in \Omega$. Here,  $a=a(x)$, $b=b(x)$ are the removal rates of the two species.  For convenience, we define
\begin{align*}
&a_{\max}:=\max_{x\in \overline{\O}}a(x),~a_{\min}:=\min_{x\in \overline{\O}}a(x),\\
&b_{\max}:=\max_{x\in \overline{\O}}b(x),~b_{\min}:=\min_{x\in \overline{\O}}b(x).
\end{align*}


We assume the following conditions to study the system \eqref{eq:main}
\begin{enumerate}[label=(B\arabic*)]  

\item \label{cond:cond1} $a=a(x),~b=b(x)\in C^{0,\alpha}(\overline{\O})\cap C^{1,0}(\overline{\O})$  for $\alpha$ in \ref{cond:cond11}. $\partial_{\textbf{n}} a|_{\partial \O}=\partial_{\textbf{n}} b|_{\partial \O}=0$ and $a>0,~b>0$.

\item  \label{cond:cond2} $H,G\in C^2(\R^+,\R^+)$, $H(0)=G(0)=0$ and $H'(z),G'(z)>0$ for any $z\geq 0$. Assume in addition that $H''(z),G''(z)<0$ for all $z>0$. 

\item \label{cond:cond3} There exists $\overline{z}>0$ such that $G\left(\dfrac{H(\overline{z})}{{a_{\min}}}\right)<b_{\min} \overline{z}$.



\end{enumerate}
Note that $H(\infty),~G(\infty)\in (0,\infty]$, denoting the limits of $H$ and $G$ as $x\to\infty$, exist. However, it is not clear whether these values are finite. This issue is crucial for estimating $\alpha_1 = H(\alpha_2)$ for some large constant $\alpha_2 > 0$, which plays an important role in the construction of super-solutions in subsequent analysis. For simplicity, we further assume that
\begin{enumerate}[label=(B\arabic*)]  
\setcounter{enumi}{3}
\item \label{cond:cond4} $H(\infty)=\infty$.
\end{enumerate}

Let us consider
\begin{align*}
&\textbf{B}:=\left(\begin{matrix}
-a&0\\
0&-b
\end{matrix}\right);~\mathcal{G}\textbf{u}:=(H(v),G(u))=\left(\begin{matrix}
0&H(v)\\
G(u)&0
\end{matrix}\right).
\end{align*}

As the first step, we study the steady state system of \eqref{eq:main}, which is given by
\begin{align}\label{eq:main1}
\begin{split}
\left\{\begin{array}{lllll}
(-d_1\Delta_N)^{s_1}u=-a(x)u+H(v),&x\in \O, \\
(-d_2\Delta_N)^{s_2}v=-b(x)v+G(u),&x\in \O.
\end{array}\right.
\end{split}
\end{align}
The key is to develop the {\it basic reproduction number} $\mathcal{R}_0$, a threshold that determines the existence of a steady state. Based on the input-output rate principle in \cite{wang_basic_2012,feng_principal_2024,feng_asymptotic_2025}, we consider the following system
\begin{align}\label{eq:linear}
\begin{split}
\left\{\begin{array}{lllll}
u_t+(-d_1\Delta_N)^{s_1}u=-a(x)u+H'(0)v,&t>0,~x\in \O, \\
v_t+(-d_2\Delta_N)^{s_2}v=-b(x)v,&t>0,~x\in \O,
\end{array}\right.
\end{split}
\end{align}
or, in the matrix form.
\begin{align*}
\partial_t\begin{pmatrix}
u\\
v
\end{pmatrix}=
\begin{pmatrix}
(-d_1\Delta_N)^{s_1}&0\\
0&(-d_2\Delta_N)^{s_2}
\end{pmatrix}\begin{pmatrix}
u\\
v
\end{pmatrix}+\begin{pmatrix}
-a(x)&H'(0)\\
0&-b(x)
\end{pmatrix}\begin{pmatrix}
u\\
v
\end{pmatrix},
\end{align*}
and, in the compact form,
\begin{align*}
{\bf u}_t = (-\textbf{d}\Delta_N)^{\textbf{s}}\textbf{u}+\mathcal{\bf D}\textbf{u},
\end{align*}
where $\mathcal{\bf D}:=\begin{pmatrix}
-a(x)&H'(0)\\
0&-b(x)
\end{pmatrix}$ and ${\bf F}:=
\begin{pmatrix}
0&0\\
G'(0)&0
\end{pmatrix}$. Following the semigroup theory for fractional Laplacian in Zhao and Ruan \cite[Section 4]{zhao_spatiotemporal_2023} and Section~\ref{section:3}, the system \eqref{eq:linear} admits a unique solution $T(t){\bf u}_0(x)$ with the initial condition ${\bf u}_0=(u_0,v_0)$. Now, we define
\begin{align*}
L({\bf u}_0)(x):=\int_0^\infty {\bf F}T(t){\bf u}_0(x)dt={\bf F}\int_0^\infty T(t){\bf u}_0(x)dt.
\end{align*}
Here, ${\bf F}T(t){\bf u}_0(x)$ represents the distribution of new infections at time $t$ and position $x$. Consequently, $L$ characterizes the distribution of the total number of new infections. By a comparison principle similar to that in Section~\ref{section:3}, it is straightforward to verify that
\begin{enumerate}
\item[•] $L:\mathcal{C}\rightarrow \mathcal{C} \text{ is bounded}.$
\item[•] $L(\textbf{u}_0)>{\bf 0}$ for any ${\bf u}_0\in \mathcal{C}$, ${\bf u}_0\geq \textbf{0}$.
\end{enumerate}
We define
\begin{align*}
\mathcal{R}_0:=r(L)\geq 0\text{, the spectral radius of $L$}.
\end{align*}
In Theorem~\ref{basic}, the relationship between the basic reproduction number 
$\mathcal{R}_0$ and the principal eigenvalue $\lambda_p$ associated with \eqref{eq:main1} 
is given by
\begin{align*}
\mathcal{R}_0>1\text{ if and only if } \lambda_p<0 \quad \text{ and } \quad \mathcal{R}_0<1\text{ if and only if } \lambda_p>0.
\end{align*}

We state existence and nonexistence results.
\begin{theorem}\label{theo:steady}
Assume \ref{cond:cond1} to \ref{cond:cond4} hold.  Then, the following statements are valid.

\begin{enumerate}[label=(\roman*)]

\item\label{Exis} If $\mathcal{R}_0>1$, system \eqref{eq:main1} admits a unique positive solution $\textbf{u}_1$.

\item\label{Non} If $\mathcal{R}_0\leq 1$, system \eqref{eq:main1} admits no non-negative non-trivial solution.

\end{enumerate}
\end{theorem}

The key step in establishing uniqueness is the application of the sliding method via the critical value $k^*$. Alternatively, when the eigenvalue $\lambda_p < 0$ and $\varphi_1$ is the associated eigenfunction, it can be shown that $\epsilon \varphi_1$ is a sub-solution of \eqref{eq:main1} for some $\epsilon>0$ small enough. In combination with the super-solution $(M_1, M_2)$ from Lemma~\ref{lemma2}, the Schauder fixed point theorem is applied to establish the existence of a non-trivial steady state. In contrast, if $\lambda_p \geq 0$, we can show that $\epsilon \varphi_1$ is a super-solution for any $\epsilon > 0$ and non-existence then follows by passing to the limit $\epsilon \to 0$.

Next, we study the well-posedness and long-time behavior of system \eqref{eq:main}.

\begin{theorem}\label{theo:existence}
Assume \ref{cond:cond1} to \ref{cond:cond4} hold. Then, for any $\textbf{u}_0=(u_0,v_0)\in \mathcal{C}\setminus \{(0,0)\}$, $u_0,v_0\geq 0$, system \eqref{eq:main} admits a unique solution $\textbf{u}(t;\textbf{u}_0)=(u(t;u_0,v_0),v(t;u_0,v_0))$ for $t>0$. In addition, $\textbf{u}(t;\textbf{u}_0)\in C(\R^+,C(\overline{\O}))$ and
\begin{align*}
(0,0)<\textbf{u}(t;\textbf{u}_0)\leq (M_1,M_2),~\forall t\geq 0\text{ on }\overline{\O}.
\end{align*}
Here, $(M_1,M_2)$ is a constant super-solution determined by Lemma~\ref{lemma2}. 
Furthermore, the following statements hold.

\begin{enumerate}[label=(\roman*)]
\item(Persistence) \label{casei} If $\mathcal{R}_0>1$, one has 
\begin{align*}
\lim_{t\rightarrow \infty}\|\textbf{u}(t;\textbf{u}_0)-\textbf{u}_1\|_{\mathcal{C}}=0,
\end{align*}
where  $\textbf{u}_1$ is the steady state in Theorem \ref{theo:steady}.

\item(Extinction) \label{caseii} If $\mathcal{R}_0<1$, one has 
\begin{align*}
\lim_{t\rightarrow \infty}\|\textbf{u}(t;\textbf{u}_0)\|_{\mathcal{C}}=0.
\end{align*}

\end{enumerate}

\end{theorem}
The theorem establishes the existence, uniqueness and global boundedness of solutions to system \eqref{eq:main} via the fractional Laplacian semigroup and the parabolic comparison principle (see Section~\ref{section:3}.2). By applying Gronwall’s inequality, the local-in-time dependence of solutions on the initial condition \eqref{continuous} is established. Then, the long-time behavior of system \eqref{eq:main} is derived by employing the sub- and super-solution technique within the framework of classical monotone dynamical systems. In particular, it is shown that if $t \mapsto \mathbf{u}(t;\mathbf{u}_0)$ is monotone, then the pointwise limit is itself a solution of steady state system \eqref{eq:main1}. Moreover, by employing the pair of sub- and super-solutions provided in Theorem~\ref{theo:steady}, the long-time behavior is established in the general case.

\textbf{We structure the paper as follows.}
Section~\ref{section:2} presents preliminary results, including the functional framework of fractional Sobolev spaces and the properties of the spectral fractional Laplacian under Neumann and Dirichlet boundary conditions. Section~\ref{Section 2.2} develops the spectral theory of the cooperative eigenvalue problem, establishing the existence and simplicity of the principal eigenvalue and analyzing its dependence on diffusion rates, fractional orders and domain variations. Section~\ref{section:3} addresses the well-posedness of the endemic system using semigroup theory and monotone dynamical systems and further develops the parabolic comparison principle for the spectral fractional Laplacian. Finally, Section~\ref{section5} investigates the existence, uniqueness and stability of steady states under conditions determined by the basic reproduction number. 



\section{\bf Preliminary results} \label{section:2}

For $s>1$, we define $[s]$ as the integer part of $s$. Then, $\mu=s-[s]$ is the fractional part of $s$. In addition, for $\textbf{u}=(u_1,u_2),~\textbf{v}=(v_1,v_2)\in \R^2$, we define
\begin{align*}
\textbf{u}\leq \textbf{v} \text{ if and only if } u_1\leq v_1,~u_2\leq v_2.
\end{align*}
We also interpret $\geq,~<,~>$ in a similar sense. For $\textbf{f}(x)=(f_1(x),f_2(x))$, $\textbf{g}(x)=(g_1(x),g_2(x))$, we denote
\begin{align*}
\textbf{f}\succeq\textbf{g} \text{ if and only if } \textbf{f}\geq \textbf{g} \text{ and } f_i\not\equiv g_i,~i=1,2.
\end{align*}

Since $H$ is a strictly concave function, it is well-known that 
\begin{align*}
H(v_1)-H(v_2)< H'(v_2)(v_1-v_2),~\forall v_1, v_2\in [0,\infty),~ v_1\neq v_2.
\end{align*}
By choosing $v_1 =0$, we can prove that
\begin{align*}
v\mapsto \dfrac{H(v)}{v} \text{ is strictly decreasing.}
\end{align*}
and, thus, $kH(v)> H(kv)$ for any $k>1,~v>0$. A similar argument applies to $G$. On the other hand, based on the mean-value theorem and $0<H'(z)<H'(0)$, $0<G'(z)<G'(0)$, one can prove that the nonlinear terms are globally Lipschitz functions. It means that, for any $u_1,u_2,v_1,v_2\in \R_+$,
\begin{align*}
&|H(u_1)-H(u_2)|\leq H'(0)|u_1-u_2|,\\
&|G(v_1)-G(v_2)|\leq G'(0)|v_1-v_2|.
\end{align*}

Next, we recall some definitions about the fractional Sobolev space and fractional Laplacian. For $0<s<1$,
\begin{align*}
H^s(\O)=\{w\in L^2(\O):\left[w\right]_{H^s(\O)}<\infty\}.
\end{align*}
where $[\cdot]_{H^s(\O)}$ is Gagliardo seminorm
\begin{align*}
\left[w\right]_{H^s(\O)}^2=\int_{\O}\int_{\O}\dfrac{|w(x)-w(y)|^2}{|x-y|^{N+2s}}dxdy.
\end{align*}
This space is equipped with the following norm
\begin{align*}
\|u\|_{H^s(\O)}=\left(\|u\|^2_{L^2(\O)}+\left[u\right]_{H^s(\O)}^2\right)^{1/2}
\end{align*}
It is well-known that $H^s(\O)$ is a Hilbert space (see \cite[Chapter 7, p. 208]{adam_sobolev_1975}). Furthermore, the embedding from $H^s(\O)$ to $L^2(\O)$ is compact (see  \cite[Theorem 4.54]{demengel_functional_2012}, \cite[Theorem 7.1]{di_hitchhiker_2012} and references therein). Also, if $w\in H^s(\O)$ then $|w|\in H^s(\O)$  since
\begin{align*}
\left|~|w(x)|-|w(y)|~\right|^2\leq |w(x)-w(y)|^2,~\forall x,y\in\O.
\end{align*}
Consequently, with $w^+:=\max\{w,0\}$ and $w^-:=-\min\{w,0\}$, one has $w^+,w^-\in H^s(\O)$ and $w=w^+-w^-$.

\subsection{\bf Neumann fractional Laplacian}

Based on the approaches in \cite{caffarelli_fractional_2016} for Neumann fractional Laplacian, with $u\in H^s(\O)$, one has 
\begin{align*}
D((-d\Delta_N)^{s})=H^s(\O);~
(-d\Delta_N)^{s}u=\sum_{k=0}^\infty d^s\mu_{N,k}^su_{k}\phi_k=\sum_{i=1}^\infty d^s\mu_{N,k}^su_k\phi_k \text{ on }L^2(\O),
\end{align*}
where $(\mu_{N,k},\phi_k)^\infty_{k=0}$ are eigenpairs of Neumann Laplacian $-\Delta_N$ and $u_k=\left<u,\phi_k\right>_{L^2}$. Here, the last equality is due to the fact that $\mu_{N,0}=0$. It is easy to see that $(-d\Delta_N)^{s}=d^s(-\Delta_N)^{s}$. Furthermore, we also have the following integro-differential formula
\begin{align*}
\left<(-d\Delta_N)^{s}u,v\right>=\int_{\O}\int_{\O}d^sK^{s,N}(x,y)[u(x)-u(y)][v(x)-v(y)]dxdy,
\end{align*}
where 
\begin{align*}
K^{s,N}(x,y):=\dfrac{1}{2|\Gamma(-s)|}\int_0^\infty \dfrac{G^N(t,x,y)}{t^{1+s}}dt.
\end{align*}
Here, $G^N(t,x,y)$ is the heat kernel of the semigroup Neumann Laplacian $e^{t\Delta_N}$. It is known that $K^{s,N}$ is symmetric and enjoys two-sided Gaussian estimates (see \cite[Section 2]{stinga_fractional_2015} and references therein). In addition,  there exists $c_1(s,\O),~c_2(s,\O)>0$ such that
\begin{align}\label{bounded2sides}
\dfrac{c_1(s,\O)}{|x-y|^{N+2s}}\leq K^{s,N}(x,y) \leq \dfrac{c_2(s,\O)}{|x-y|^{N+2s}}.
\end{align}
This approach is discussed in \cite[Section 7]{caffarelli_fractional_2016}.


\subsection{\bf Dirichlet fractional Laplacian}
Based on the approaches in \cite{caffarelli_fractional_2016} for Dirichlet fractional Laplacian, with $u\in H^s(\O)$, one has 
\begin{align*}
D((-d\Delta_D)^{s})=\mathbb{H}^s(\O);~
(-d\Delta_D)^{s}u=\sum_{k=0}^\infty d^s\mu_{D,k}^su_k\phi_k \text{ on }L^2(\O),
\end{align*}
where $(\mu_{D,k},\phi_k)^\infty_{k=0}$ are eigenpairs of Dirichlet Laplacian $-\Delta_D$ and $u_k=\left<u,\phi_k\right>_{L^2}$. Note that $(-d\Delta_D)^{s}=d^s(-\Delta_D)^{s}$. We define the following inner product
\begin{align}\label{eigen:ex}
(u, \psi)_{\mathbb{H}^s} := \sum_{k=0}^\infty \mu_{D,k}^s u_k \psi_k,
\end{align}
where $\psi = \sum\limits_{k=0}^\infty \psi_k \phi_k \in \mathbb{H}^s$. Observe that, for some positive constant $C$, we have 
\[
\|u\|_{L^2(\Omega)}^2 \leq C (u, u)_{\mathbb{H}^s} = C \|u\|_{\mathbb{H}^s}^2,
\] 
so that $(\cdot, \cdot)_{\mathbb{H}^s}$ defines indeed an inner product in $\mathbb{H}^s$. For $u \in \mathbb{H}^s$, we can view $(-d\Delta_D)^{s}$ as follows
\[
\langle (-d\Delta_D)^{s} u, \psi \rangle = \sum_{k=0}^\infty \mu_{D,k}^s u_k d_k = (u, \psi)_{\mathbb{H}^s},
\]
where $\langle \cdot, \cdot \rangle$ denotes the pairing between $\mathbb{H}^s$ and $\mathbb{H}^{-s}$. 
Furthermore, we also have the following integro-differential formula
\begin{align*}
\left<(-d\Delta_D)^{s}u,v\right>=\int_{\O}\int_{\O}d^sK^{s}(x,y)[u(x)-u(y)][v(x)-v(y)]dxdy+\int_{\O}d^s u(x)v(x)B^s(x)dx,
\end{align*}
where $u,v\in \mathbb{H}^s$ and $K^{s}$ and $B^{s}$ are given in \cite[Theorem 2.3]{caffarelli_fractional_2016}.

\begin{remark}\label{remark1}
We use the following notation
\[
\mathbb{H}^s(\Omega) :=
\begin{cases}
H^s(\Omega), & \text{when } 0 < s < \tfrac{1}{2}, \\
H^{1/2}_{00}(\Omega), & \text{when } s = \tfrac{1}{2}, \\
H^s_0(\Omega), & \text{when } \tfrac{1}{2} < s < 1.
\end{cases} 
\]
The spaces $H^s(\Omega)$ and $H^s_0(\Omega)$, $s \neq 1/2$, are the classical fractional Sobolev spaces given by the completion of $C_c^\infty(\Omega)$ under the norm $\|u\|_{H^s(\Omega)}$. The space $H^{1/2}_{00}(\Omega)$ is the Lions--Magenes space which consists of functions $u \in L^2(\Omega)$ such that $[u]_{H^{1/2}(\Omega)} < \infty$ and
\[
\int_\Omega \frac{u(x)^2}{\operatorname{dist}(x, \partial \Omega)} \, dx < \infty.
\]
See \cite[Chapter 1]{lions_pro_1968}, also \cite[Section 2]{nochetto_a_2015},  for a discussion. The norm in any of these spaces is denoted by $\|\cdot\|_{H^s}$.

\end{remark}

%
%

\section{\bf Eigenvalue problem}\label{Section 2.2}

In this section, we study the eigenvalue problem \eqref{eq:eigen}. The asymptotic behavior of the principal eigenvalue is then investigated.

We first study the existence of the principal eigenvalue.

\begin{proof}[Proof of Theorem \ref{eigen_theo}]

The proof is divided by several steps as follows.

{\bf Step 1:} Existence of a simple, principal eigenvalue $\lambda_p$ and $\lambda_p^D$.

We only need to prove the result for the Neumann case, as the Dirichlet case is similar. Define ${\bf A_0}=-{\bf A}$. For any $\beta>0$ and $\lambda \in [-\beta,\infty)$, consider
\begin{align*}
K_{\lambda,\beta}\textbf{u}=((-\textbf{d}\Delta_N)^{\textbf{s}}+\beta I)^{-1}[{\bf A_0u}+(\lambda+\beta)\textbf{u}],
\end{align*}
for $\textbf{u}\in C(\overline{\O})$. Following the proof of \cite[Theorem 3.1, Steps 1–2]{zhao_spatiotemporal_2025}, one sees that $\big((-\mathbf{d}\Delta_N)^{\mathbf{s}}+\beta I\big)^{-1}$ is compact on $C(\overline{\Omega})$ and, moreover, there exists $K > 0$ such that
\begin{align*}
\left\|((-\textbf{d}\Delta_N)^{\textbf{s}}+\beta I)^{-1}\right\|\leq K \beta^{-1}.
\end{align*}
By the positivity result in \cite[Theorem 3.1, Step~3]{zhao_spatiotemporal_2025}, it follows from the Krein–Rutman theorem that the spectral radius $r(K_{\lambda,\beta}) > 0$ is the principal eigenvalue of $K_{\lambda,\beta}$ and
\begin{align*}
r(K_{\lambda,\beta})=\lim_{m\rightarrow \infty}\|(K_{\lambda,\beta})^m\|^{1/m}.
\end{align*}
Furthermore, if $\lambda<\lambda'$, then $K_{\lambda,\beta}\textbf{u}<K_{\lambda',\beta}\textbf{u}$ for any $\textbf{u}\in \mathcal{C}$ and $\mathbf{u}>0$. By applying \cite[Theorem 3.2(v)]{amann_fixed_1976}, one has $r(K_{\lambda,\beta})<r(K_{\lambda',\beta})$. In addition, based on the analytic perturbation theory in \cite[Chapter 7, Theorem 1.8]{kato_1976}, it is known that
\begin{align*}
\lambda \mapsto r(K_{\lambda,\beta}) \text{ is analytic for } \lambda \in [-\beta,\infty).
\end{align*}
Hence, it is continuous on $[-\beta,\infty)$. 

Next, in view of \cite[Lemma 3.2]{lam_ays_2016}, we  prove the existence of the principal eigenvalue of \eqref{eq:eigen}.
Consider $K_{-\beta,\beta}$, one can prove that
\begin{align*}
r(K_{-\beta,\beta})\leq \|K_{-\beta,\beta}\|\leq K\|{\bf A_0}\|_2\beta^{-1}<1.
\end{align*}
whenever $\beta>0$ large enough. On the other hand, with large enough $\beta>0$ so that $a_{ii}+\beta>0$ on $\overline{\O}$ for any $i=1,2$, one has
\begin{align*}
K_{\lambda,\beta}\textbf{u}\geq\lambda((-\textbf{d}\Delta_N)^{\textbf{s}}+\beta I)^{-1}\textbf{u},~\forall \textbf{u}\in \mathcal{C},~\textbf{u}\geq {\bf 0}.
\end{align*}
This implies that $r(K_{\lambda,\beta})\geq \lambda r(((-\textbf{d}\Delta_N)^{\textbf{s}}+\beta I)^{-1})$. Thus, with $\lambda$ large enough, one has $r(K_{\lambda,\beta})>1$. By continuity, there exists $-\beta<\lambda_0$ such that
\begin{align*}
r(K_{\lambda_0,\beta})=1.
\end{align*} 
Consequently, one can prove that $\lambda_0=\lambda_p$ is indeed the principal eigenvalue associated with an eigenfunction $\varphi_1 \in\mathcal{C}$, $\varphi_1>0$ on $\overline{\O}$ and $Re(\lambda)>\lambda_p$ for any other eigenvalue $\lambda$ of \eqref{eq:eigen}. Furthermore, it is independent of $\beta$ due to the uniqueness in Krein-Rutman theorem and it is simplicity. In addition, if there exists another eigenvalue $\lambda$ with a positive eigenfunction, then necessarily $\lambda = \lambda_p$. 

We now establish the following claim, which will be useful in subsequent analysis.
\phantom{1}

\textbf{Claim:}  The following statements hold.
\begin{enumerate}[label=(\roman*)]
\item\label{<1}  $r(K_{\lambda,\beta})=1$ if and only if $\lambda = \lambda_p$.
\item\label{>1}  $r(K_{\lambda,\beta})<1$ if and only if $\lambda < \lambda_p$.

\item\label{=1} $r(K_{\lambda,\beta})>1$ if and only if $\lambda > \lambda_p$.

\end{enumerate}

\begin{proof}

Consider $r(K_{\lambda,\beta})<1=r(K_{\lambda_p,\beta})$. If $\lambda\geq \lambda_p$, then $r(K_{\lambda,\beta})\geq r(K_{\lambda_p,\beta})$, which is a contradiction. This concludes the proof.

\end{proof}

Note that the existence of the Dirichlet principal eigenvalue $\lambda^D_p$ can be established in a similar manner on $[C_0(\overline{\Omega})]^2$, 
where $C_0(\overline{\Omega})=\{v\in C(\overline{\Omega}):v|_{\partial \Omega}=0\}$ 
(see \cite[Theorem~3.1, Steps~1--3]{zhao_spatiotemporal_2025} for all relevant properties).

From \eqref{eq:eigen}, for some $K>0$ large enough, it is known that
\begin{align*}
((-\textbf{d}\Delta_N)^{\textbf{s}} + 2K I)\varphi_1=(\lambda_p+K)\varphi_1+(K-\textbf{A})\varphi_1.
\end{align*}
or
\begin{align*}
\left\{\begin{array}{llll}
((-d_1\Delta_N)^{s_1} + 2K I)\varphi_{1,1}=\left(K-a_{11}+(\lambda_p+K)\right)\varphi_{1,1}-a_{12}\varphi_{1,2},\\
((-d_2\Delta_N)^{s_2} + 2K I)\varphi_{1,2}=-a_{21}\varphi_{1,1}+\left(K-a_{22}+(\lambda_p+K)\right)\varphi_{1,2}.\\
\end{array}\right.
\end{align*}
One can check that
\begin{align*}
&\textbf{G}:=(g_1,g_2):=\left(\left(K-a_{11}+(\lambda_p+K)\right)\varphi_{1,1}+(-a_{12})\varphi_{1,2} ,(-a_{21})\varphi_{1,1}+\left(K-a_{22}+(\lambda_p+K)\right)\varphi_{1,2}\right)\succeq  {\bf 0}.
\end{align*}
Thus, in the weak sense,
\begin{align*}
((-\textbf{d}\Delta_N)^{\textbf{s}} + 2K I)\varphi_1=\textbf{G}.
\end{align*}
or 
\begin{align*}
\begin{split}
\left\{\begin{array}{lllll}
(-d_1\Delta_N)^{s_1}\varphi_{1,1}+2K\varphi_{1,1}=g_1,&x\in \O, \\
(-d_2\Delta_N)^{s_2}\varphi_{1,2}+2K\varphi_{1,2}=g_2,&x\in \O.
\end{array}\right.
\end{split}
\end{align*}
By similar arguments in \cite[Theorem 3.1, Step 3, Case 1]{zhao_spatiotemporal_2025}, for $i=1,2$, one can show that
\begin{align*}
((-d_i\Delta_N)^{s_i}+2KI)^{-1} \text{ exists and strongly positive on }\mathcal{C},
\end{align*}
and there exists $\epsilon>0$ such that
\begin{align*}
((-d_i\Delta_N)^{s_i}+2KI)^{-1}g_i\geq \epsilon\text{ on } \overline{\O}.
\end{align*}
This leads to $\varphi_{1}\geq (\epsilon,\epsilon)>\textbf{0}$ on $\overline{\O}$.

For the Dirichlet case, one has $\varphi_{1}>(\epsilon,\epsilon)\dist(x,\partial \O)>\textbf{0}$ on $\O$.

{\bf Step 2:} Regularity of the eigenfunction.

It is easy to verify that
\begin{align*}
\dfrac{\partial g_1}{\partial\textbf{n}}=-\dfrac{\partial a_{11}}{\partial\textbf{n}}+(\lambda_1-a_{11})\dfrac{\partial\varphi_{1,1}}{\partial\textbf{n}}-\dfrac{\partial a_{12}}{\partial\textbf{n}}-a_{12}\dfrac{\partial \varphi_{1,2}}{\partial\textbf{n}} =0,
\end{align*}
in trace sense. Similar manner to $g_2$.  By applying \cite[Theorem 4.1]{grubb_regularity_2016} and \cite[Corollary 2.2]{zhao_spatiotemporal_2025}, a classical bootstrap argument can be used to obtain that
\begin{align*}
\varphi_{1}\in C^{2\underline{s}+2\alpha}(\overline{\O})\times C^{2\underline{s}+2\alpha}(\overline{\O})\cap \mathcal{H}^{\bf s}.
\end{align*}
 The bootstrap argument first raises the regularity of $\varphi_1$ to a certain $s$. Further regularity then requires compatibility conditions of $(f_1,f_2)$ for the Neumann boundary condition.
 


In a similar manner, for Dirichlet case, it follows from \cite[Corollary 2.3]{zhao_spatiotemporal_2025} that
\begin{align*}
\varphi_{1}\in C_0^{2\underline{s}+2\alpha}(\overline{\O})\times C_0^{2\underline{s}+2\alpha}(\overline{\O})\cap \mathbb{H}^{s_1}\times \mathbb{H}^{s_2} .
\end{align*}
where $C_0^{2\underline{s}+2\alpha}(\overline{\O})=\left\{ u\in C^{2\underline{s}+2\alpha}(\overline{\O}): u|_{\partial \O}=0\right\}$.


\textbf{Step 3:} Suppose that $a_{12}=a_{21}$. We study the variational method for $\lambda_p$ and $\lambda_p^D$.



 To this end, first, for $K>0$ large enough, it follows from Lax-Milgram theorem on Hilbert space $\mathcal{H}^{\bf s}$ that 
\begin{align*}
((-\textbf{d}\Delta_N)^{\textbf{s}} + \textbf{A}+K I_{\mathcal{H}^{\textbf{s}}})^{-1}: \mathcal{X}&\rightarrow \mathcal{X}\\
f&\mapsto u
\end{align*}
is well-defined, bounded and compact thanks to the fact that  the embedding of $\mathcal{H}^{\textbf{s}}$ to $\mathcal{X}$ is compact. Similar to $((-\textbf{d}\Delta_D)^{\textbf{s}} + \textbf{A}+K I_{\mathcal{H}^{\textbf{s}}})^{-1}$ on $\mathbb{H}^{s_1}\times\mathbb{H}^{s_2}$. Then, by applying \cite[Theorem 6.11]{brezis_functional_2011} and Fredholm alternative theorem, one has
\begin{align*}
\sigma\left(\mathcal{K}_B\right)= \sigma((-\textbf{d}\Delta_B)^{\textbf{s}}+ \textbf{A}))\setminus\{\infty\}=\{\lambda^B_{n}\}_{n\in \mathbb{N}},
\end{align*}
which is the spectrum of $(-\mathbf{d}\Delta_B)^{\mathbf{s}}+\mathbf{A}$ on $\mathcal{H}^{\mathbf{s}}$ for $B=N$ and on $\mathbb{H}^{s_1}\times\mathbb{H}^{s_2}$ for $B=D$. Here, we have 
\begin{align*}
\lambda_1^B<\lambda_2^B<...<\lambda_n^B\rightarrow \infty \text{ as }n\rightarrow \infty.
\end{align*}

Next, we define the variational function as follows
\begin{align*}
\mathcal{J}(\textbf{u}):=\left<\mathcal{K}_B{\bf u},{\bf u}\right>=\left<(-\textbf{d}\Delta_B)^{\textbf{s}}{\bf u},{\bf u}\right> + \left<{\bf Au},{\bf u}\right>;\quad ~\lambda_*:=\inf\{ \mathcal{J}(\textbf{u}):\textbf{u}\in \mathcal{D}(\mathcal{J}),~\|\textbf{u}\|_{\mathcal{X}}=1\}.
\end{align*}
where $\mathcal{D}(\mathcal{J})=\mathcal{H}^{\bf s}$ if $B=N $ and $\mathcal{D}(\mathcal{J})=\mathbb{H}^{s_1}\times\mathbb{H}^{s_2}$ if $B=D$. By the standard variational method via the weak convergence in a reflexive space $\mathcal{H}^{\bf s}$ (see, for instance, \cite[Appendix A]{servadei_varational_2013}), one can show that $\lambda_*$ is an eigenvalue and the eigenfunctions associated with $\lambda_*$ are minimizers of $\mathcal{J}$.  Thus, there exists $n\in \mathbb{N}$ such that $\lambda_*=\lambda_n^B\geq \lambda_1^B$. On the other hand, one has
\begin{align*}
\lambda_*\leq \mathcal{J}(\varphi_1)=\left<\mathcal{K}_B\varphi_1,\varphi_1\right>=\lambda_1^B.
\end{align*}
Hence, $\lambda_1^B = \min\{ \mathcal{J}(\textbf{u}):\textbf{u}\in \mathcal{D}(\mathcal{J}),~\|\textbf{u}\|_{\mathcal{X}}=1\}$. Note that $\lambda_1^B$ is the smallest eigenvalue. Hence, it follows that $\lambda_p=\lambda_1^N$ and $\lambda_p^D=\lambda_1^D$, the principal eigenvalue in Step 1.

We conclude the proof.
\end{proof}

Next, let us provide a useful comparison for the principal eigenvalue.

\begin{corollary}
Assume \ref{cond:cond11} and \ref{cond:cond21} hold and $a_{12}=a_{21}$. The following statements hold
\begin{enumerate}[label=(\roman*)]
\item\label{ismaller}  Suppose there exist $\lambda\in\R$ and $\boldsymbol{\phi}=(\phi_1,\phi_2)\in \mathcal{C}\cap \mathcal{H}^{\bf s}$,~$\boldsymbol{\phi}\succeq {\bf 0}$ such that
\begin{align*}
\mathcal{K}_B{\bf \boldsymbol{\phi}}-\lambda\boldsymbol{\phi} \geq  {\bf 0}.
\end{align*}
Then, $\lambda_p\geq \lambda$ in Neumann case $B=N$ and $\lambda_p^D\geq \lambda$ in Dirichlet case $B=D$. 

\item\label{ibigger}  Suppose there exist $\lambda\in\R$ and $\boldsymbol{\phi}=(\phi_1,\phi_2)\in \mathcal{C}\cap \mathcal{H}^{\bf s}$,~$\boldsymbol{\phi}\succeq {\bf 0}$ such that
\begin{align*}
\mathcal{K}_B{\bf \boldsymbol{\phi}}-\lambda\boldsymbol{\phi} \leq  {\bf 0}.
\end{align*}
Then, $\lambda_p\leq \lambda$ in Neumann case $B=N$ and $\lambda_p^D\leq \lambda$ in Dirichlet case $B=D$. Furthermore, the equality is achieved if and only if $\boldsymbol{\phi}$ is the eigenfunction of \eqref{eq:eigen}.

\item\label{generalized} Recall the following generalized principal eigenvalues $\lambda_p'\left(\mathcal{K}_B\right)$ and $\lambda_p''\left(\mathcal{K}_B\right)$
\begin{align*}
&\lambda_p'\left(\mathcal{K}_B\right) := \inf\left\{\lambda\in \R:\exists \boldsymbol{\phi}\in \mathcal{C}\cap \mathcal{H}^{\bf s},~\boldsymbol{\phi}>{\bf 0},~\mathcal{K}_B{\bf \boldsymbol{\phi}}-\lambda\boldsymbol{\phi} \leq  {\bf 0} \right\},\\
&\lambda_p''\left(\mathcal{K}_B\right) := \sup\left\{\lambda\in \R:\exists \boldsymbol{\phi}\in \mathcal{C}\cap \mathcal{H}^{\bf s},~\boldsymbol{\phi}>{\bf 0},~\mathcal{K}_B {\bf \boldsymbol{\phi}}-\lambda\boldsymbol{\phi} \geq  {\bf 0} \right\}.
\end{align*}
Then, one has 
\begin{align*}
\lambda_p\left(\mathcal{K}_B\right)=\lambda_p'\left(\mathcal{K}_B\right)=\lambda_p''\left(\mathcal{K}_B\right).
\end{align*}

\end{enumerate}
\end{corollary}

\begin{proof} \ref{ismaller} For simplicity, we only need to prove for Neumann case $B=N$; the Dirichlet case $B=D$ is in a similar manner. Recall that $\boldsymbol{\phi}\succeq {\bf 0}$ means that $\phi_i\geq 0$ and $\phi_i\not\equiv 0$ for $i=1,2$.

Now, let $\varphi_1$ be a positive eigenfunction associated with $\lambda_p$. Then, one has
\begin{align*}
\left<\mathcal{K}_N\boldsymbol{\phi},\boldsymbol{\phi}\right>= \left<(-\textbf{d}\Delta_N)^{\textbf{s}}\boldsymbol{\phi} + {\bf A\boldsymbol{\phi}}-\lambda\boldsymbol{\phi},\varphi_1\right>\geq 0.
\end{align*}	
This leads to
\begin{align*}
\lambda_p\left<\varphi_1,\boldsymbol{\phi}\right>=\left<(-\textbf{d}\Delta_N)^{\textbf{s}}\varphi_1 + {\bf A}\varphi_1,\boldsymbol{\phi}\right>=\left<(-\textbf{d}\Delta_N)^{\textbf{s}}\boldsymbol{\phi} + {\bf A\boldsymbol{\phi}},\varphi_1\right>\geq \lambda\left<\boldsymbol{\phi},\varphi_1\right>,
\end{align*}
which implies that $\lambda_p\geq \lambda$. If one of these inequality is strict, we can show that $\lambda_p>\lambda$.

\ref{ibigger} The proof is a similar manner. We omit the proof here.

\ref{generalized} Next, consider $\lambda>\lambda'_p\left(\mathcal{K}_N\right)$ and there exists $\boldsymbol{\phi}\in \mathcal{C}\cap \mathcal{H}^{\bf s},~\boldsymbol{\phi}>{\bf 0}$ such that
\begin{align*}
(-\textbf{d}\Delta_N)^{\textbf{s}}\boldsymbol{\phi} + {\bf A\boldsymbol{\phi}}\leq \lambda\boldsymbol{\phi}.
\end{align*}
Thanks to \ref{ibigger}, we get $\lambda_p\leq\lambda $. This implies that $\lambda_p\leq \lambda'_p\left(\mathcal{K}_N\right)$.

On the other hand, if $\lambda<\lambda''_p\left(\mathcal{K}_N\right)$, one has
\begin{align*}
(-\textbf{d}\Delta_N)^{\textbf{s}}\boldsymbol{\phi} + {\bf A\boldsymbol{\phi}}\geq \lambda\boldsymbol{\phi}.
\end{align*}
for some $\boldsymbol{\phi}\in \mathcal{C}\cap \mathcal{H}^{\bf s},~\boldsymbol{\phi}>{\bf 0}$. It follows from  \ref{ismaller} that $\lambda_p\geq \lambda$. Hence, one has
\begin{align*}
\lambda_p''\left(\mathcal{K}_N\right)\leq \lambda_p\leq \lambda_p'\left(\mathcal{K}_N\right).
\end{align*}

Alternatively, it follows from the following eigenvalue-eigenvector of $\lambda_p$
\begin{align*}
(-\textbf{d}\Delta_N)^{\textbf{s}}\varphi_1 + {\bf A}\varphi_1 = \lambda_p \varphi_1
\end{align*}
that $\lambda_p\in \left\{\lambda\in \R:\exists \boldsymbol{\phi}\in \mathcal{C}\cap \mathcal{H}^{\bf s},~\boldsymbol{\phi}>{\bf 0},~\mathcal{K}_B{\bf \boldsymbol{\phi}}-\lambda\boldsymbol{\phi} \leq  {\bf 0} \right\}\cap \left\{\lambda\in \R:\exists \boldsymbol{\phi}\in \mathcal{C}\cap \mathcal{H}^{\bf s},~\boldsymbol{\phi}>{\bf 0},~\mathcal{K}_B{\bf \boldsymbol{\phi}}-\lambda\boldsymbol{\phi} \geq  {\bf 0} \right\}$. This implies that
\begin{align*}
\lambda_p'\left(\mathcal{K}_N\right)\leq \lambda_p\leq \lambda_p''\left(\mathcal{K}_N\right).
\end{align*}
We obtain the desired results.
	
\end{proof}


\subsection{\bf Maximum principle}

We establish the maximum principle for the operator $\mathcal{K}_B=(-\textbf{d}\Delta_B)^{\textbf{s}}\textbf{u} + \textbf{Au} $ in this section. We further develop a counterexample when the maximum principle fails.


\begin{theorem}[Weak maximum principle]
Assume \ref{cond:cond11} and \ref{cond:cond21} hold and $a_{12}=a_{21}$. Then the following two statements are equivalent.
\begin{enumerate}[label=(\roman*)]
\item\label{positivity} $\lambda_p>0$

\item\label{WKP} If ${\bf u}\in \mathcal{C}\cap \mathcal{H}^{\bf s}$ and $ \mathcal{K}_N{\bf u}\geq {\bf 0} \text{ on }\O,$ then ${\bf u}\geq {\bf 0}$ on $\overline{\O}$.

\end{enumerate}
The result also holds for Dirichlet case on $\mathcal{C}\cap \mathbb{H}^{s_1}\times \mathbb{H}^{s_2}$.

\end{theorem}

\begin{proof} \ref{positivity} implies \ref{WKP}. Following the work of \cite[Theorem 2.1]{zhao_principal_2021}, we define ${\bf u}^-=(u_1^-,u_2^-)$, ${\bf u}^+=(u_1^+,u_2^+)$ and ${\bf f}=\mathcal{K}_N\textbf{u}$. It follows that
\begin{align*}
0&= \left<(-\textbf{d}\Delta_N)^{\textbf{s}}\textbf{u}+ {\bf Au}-{\bf f},{\bf u}^-\right>\\
&=\left<(-\textbf{d}\Delta_N)^{\textbf{s}}\textbf{u}^++{\bf A}{\bf u}^+, {\bf u}^-\right>-\left<(-\textbf{d}\Delta_N)^{\textbf{s}}\textbf{u}^-+{\bf A}{\bf u}^-, {\bf u}^-\right>-\left<{\bf f},{\bf u}^-\right>.
\end{align*}
It is known that
\begin{align*}
\left<(-\textbf{d}\Delta_N)^{\textbf{s}}\textbf{u}^- + {\bf A}{\bf u}^-,{\bf u}^-\right>\geq \lambda_p\|{\bf u}^-\|_{\mathcal{X}}^2,
\end{align*}
and
\begin{align*}
&\left<(-\textbf{d}\Delta_N)^{\textbf{s}}\textbf{u}^+ + {\bf A}{\bf u}^+,{\bf u}^-\right>\\
&=\int_{\O}\int_{\O}d_1^{s_1}K^{s_1}(x,y)[u^+_1(x)-u^+_1(y)][u^-_1(x)-u^-_1(y)]dxdy\\
&+\int_{\O}\int_{\O}d_2^{s_2}K^{s_2}(x,y)[u^+_2(x)-u^+_2(y)][u^-_2(x)-u^-_2(y)]dxdy+\left<{\bf A}{\bf u}^+,{\bf u}^-\right>\\
&=-2\int_{\O}\int_{\O}d_1^{s_1}K^{s_1}(x,y)u^+_1(x)u^-_1(y)dxdy-2\int_{\O}\int_{\O}d_2^{s_2}K^{s_2}(x,y)u^+_2(x)u^-_2(y)dxdy+\left<{\bf A}{\bf u}^+,{\bf u}^-\right>.
\end{align*}
Note that since $u^+_1u^-_2\geq 0$ and $a_{12},a_{21}<0$
\begin{align*}
\left<{\bf A}{\bf u}^+,{\bf u}^-\right>=\int_{\O}a_{12}u_1^+u^-_2dx + \int_{\O}a_{21}u_2^+u^-_1dx\leq 0,\quad \left<{\bf f},{\bf u}^-\right>\geq 0.
\end{align*}
As the results, one has
\begin{align*}
0\leq -\lambda_p\|{\bf u}^-\|_{\mathcal{X}}^2,
\end{align*}
which implies ${\bf u}^-\equiv 0$. Thus, ${\bf u}\geq {\bf 0}$.

\ref{WKP} implies \ref{positivity}. Suppose otherwise $\lambda_p\leq 0$. Consider the principal eigenfunction $\varphi_1>{\bf 0}$ such that
\begin{align*}
\mathcal{K}_N\varphi_1=\lambda_p \varphi_1.
\end{align*}
Define ${\bf v} = -\varphi_1<{\bf 0}$ on $\O$, one has $\mathcal{K}_N{\bf v}=-\lambda_p \varphi_1\geq {\bf 0}.$ It follows from \ref{WKP} that ${\bf v} \geq {\bf 0}$ on $\O$. Contradiction. We obtain the desired result.

\end{proof}

%


%

Next, we consider the strong maximum principle.

\begin{definition}
We say $\mathcal{K}_N$ satisfies {\it strong maximum principle} if
\begin{center}
${\bf u}\in \mathcal{C}\cap \mathcal{H}^{\bf s}$ and $ \mathcal{K}_N{\bf u}\succeq {\bf 0} \text{ on }\O$ implies that ${\bf u}> {\bf 0}$ on $\overline{\O}$.
\end{center}
Similarly, we say $\mathcal{K}_D$ satisfies {\it strong maximum principle} if
\begin{center}
${\bf u}\in \mathcal{C}\cap \mathbb{H}^{s_1}\times \mathbb{H}^{s_2}$ and $ \mathcal{K}_D{\bf u}\succeq {\bf 0} \text{ on }\O$ implies that ${\bf u}> {\bf 0}$ on $\O$.
\end{center}

\end{definition}

We have to make some comments as follows: The case $\lambda_p=0$ is quite delicate. In this case, one has $\mathcal{K}_N\varphi_1={\bf 0},$ where $\varphi_1>{\bf 0}$ is the positive principal eigenfunction. Define
\begin{align*} 
\mathcal{V}=\{{\bf v}\in \mathcal{C}\cap \mathcal{H}^{\bf s}:\mathcal{K}_N{\bf v}\succeq {\bf 0}\}
\end{align*}
If $\mathcal{V}\neq \emptyset$, we consider ${\bf u}\in \mathcal{V}$. Define ${\bf g}=\mathcal{K}_N{\bf u}\succeq {\bf 0}$. Then, one has $\left<{\bf g},\varphi_1\right>_{\mathcal{X}}>0$. On the other hand, 
\begin{align*}
\left<{\bf g},\varphi_1\right>_{\mathcal{X}}=\left<\mathcal{K}_N{\bf u},\varphi_1\right>_{\mathcal{X}}=\left<{\bf u},\mathcal{K}_N\varphi_1\right>_{\mathcal{X}}=0.
\end{align*}
Contradiction. Thus, $\mathcal{V}=\emptyset$. This indicates that the definition of the strong maximum principle might be vacuously satisfied in the borderline case $\lambda_p=0$. Such a situation is rather unnatural from the analytic viewpoint, since one expects the strong maximum principle to capture genuine positivity of nontrivial supersolutions rather than hold trivially. Hence, we suggest formulating the strong maximum principle based on \cite[Proposition A.7]{zhao_spatiotemporal_2025} as follows.

\begin{theorem}[Strong maximum principle]
Assume \ref{cond:cond11} and \ref{cond:cond21} hold and $a_{12}=a_{21}$. Then the following two statements are equivalent.
\begin{enumerate}[label=(\roman*)]
\item\label{SMP} $\lambda_p > 0$.

\item\label{SMP1} $\mathcal{K}_N$ satisfies the {\it strong maximum principle} and for any ${\bf f}\in \mathcal{X}$, there exists a unique $u=u(f)\in  \mathcal{X}$ satisfies
\begin{align*}
\mathcal{K}_N u = {\bf f}.
\end{align*}

\end{enumerate}
The similar manner holds for $\mathcal{K}_D$ and $\lambda^D_p$.

\end{theorem}

\begin{proof}  \ref{SMP} implies \ref{SMP1}. In the case $\lambda_p>0$, the weak maximum principle holds. Next, based on \cite[Theorem 3.1, Step 3, Case 1]{zhao_spatiotemporal_2025},   one has
\begin{align*}
((-{\bf d}\Delta_N)^{\bf s}+\beta I)^{-1}=\mathrm{diag}(((-d_1\Delta_N)^{s_1}+\beta I)^{-1},((-d_2\Delta_N)^{s_2}+\beta I)^{-1}) =:R_\beta, 
\end{align*}
provided $\beta$ large enough and $R_\beta{\bf f}>0$ on $\overline{\O}$ for any ${\bf f}=(f_1,f_2)\in \mathcal{C}$, ${\bf f}\succeq {\bf 0}$. It is known that
\begin{align*}
 {\bf 0}\leq \mathcal{K}_N{\bf u}= (-{\bf d}\Delta_N)^{\bf s}{\bf u} +{\bf A}{\bf u}=((-{\bf d}\Delta_N)^{\bf s}+\beta I){\bf u} +({\bf A}-\beta I){\bf u}.
\end{align*}
Define ${\bf g}=(g_1,g_2)=(\beta I-{\bf A}){\bf u}$. For $\beta$ large enough, it follows from the weak maximum principle in Theorem 7 that ${\bf g}\succeq {\bf 0}$. This implies that
\begin{align*}
{\bf h}:=((-{\bf d}\Delta_N)^{\bf s}+\beta I){\bf u}  \geq {\bf g}\geq {\bf 0},
\end{align*}
which leads to ${\bf u}=((-{\bf d}\Delta_N)^{\bf s}+\beta I)^{-1}{\bf h}>0$ on $\overline{\O}$.

Next, with ${\bf f}\in \mathcal{X}$, we consider the following functional 
\begin{align*}
I({\bf u})=\left<\mathcal{K}_N{\bf u},{\bf u}\right>-\left<{\bf f},{\bf u}\right>,~\forall {\bf u} \in \mathcal{D}(\mathcal{J}).
\end{align*}
Then, one has
\begin{align*}
DI({\bf u})({\bf v})=\left<\mathcal{K}_N{\bf u},{\bf v}\right>-\left<{\bf f},{\bf v}\right>~\forall {\bf u},~{\bf v} \in \mathcal{D}(\mathcal{J}).
\end{align*}
Direct calculation yields that
\begin{align*}
[DI({\bf u_1})-DI({\bf u_2})]({\bf u_1}-{\bf u_2})=\left<\mathcal{K}_N{\bf u_1}-\mathcal{K}_N{\bf u_2},{\bf u_1}-{\bf u_2}\right>\geq \lambda_p\|{\bf u_1}-{\bf u_2}\|_{\mathcal{X}}^2>0,~\forall {\bf u_1}\neq {\bf u_2}.
\end{align*}
Namely, $I$ is strictly convex. Thus, it is standard that there exist a unique $u=u(f)$ such that 
\begin{align*}
DI({\bf u})({\bf v})=\left<\mathcal{K}_N{\bf u},{\bf v}\right>-\left<{\bf f},{\bf v}\right>=0.
\end{align*}

 \ref{SMP1} implies \ref{SMP}. Suppose $\lambda_p<0$. Define ${\bf v}=-\varphi_1$, one can check that $\mathcal{K}_N{\bf v}=-\lambda_p\varphi_1>{\bf 0}$. Contradiction.

On the other hand, in the case $\lambda_p=0$. Then, $\mathcal{K}_N\varphi_1={\bf 0},$ where $\varphi_1>{\bf 0}$ is the positive principal eigenfunction. Since $\mathcal{K}_N{\bf 0}={\bf 0}$,  \ref{SMP1} implies that  $\varphi_1\equiv {\bf 0}$. Contradiction. consequently, $\lambda_p>0$.

\end{proof}

\subsection{\bf Effects of parameters on the principal eigenvalue}

Let us first introduce the classical theorem for matrices as follows 
\begin{theorem}[Perron–Frobenius Theorem \cite{gantmacher_theory_1959}]\label{Matrixtheo}
Given a real-valued square matrix ${\bf A}=(a_{ij})_{N\times N }$, whose off-diagonal terms are non-negative, (i.e. $a_{ij} \geq 0$ if $i\neq j$), there exists a real eigenvalue $\overline{\lambda}({\bf A})$, corresponding to a non-negative eigenvector, with the greatest real part (for any eigenvalue $\lambda$ of ${\bf A}$, one has $\overline{\lambda}({\bf A})>Re( \lambda)$). Moreover, if $a_{ij} > 0$ if $i\neq j$, then $\overline{\lambda}({\bf A})$ is simple with strictly positive eigenvector and it can be characterized as the unique eigenvalue corresponding to a nonnegative vector.
\end{theorem}

Recall that with $\mathbf{A}_0 = -\mathbf{A}$, the principal eigenvalue of $\mathbf{A}$ can be expressed as
\begin{align*}
\overline{\lambda}(\mathbf{A}) = -\overline{\lambda}(\mathbf{A}_0),
\end{align*}
which inherits properties analogous to those in Theorem~\ref{Matrixtheo}, except that $\overline{\lambda}(\mathbf{A}) < \lambda$ for any other eigenvalue $\lambda$ of $\mathbf{A}$. We further choose the principal eigenfunction $\phi_1(x) \geq (0,0)$ associated with $\overline{\lambda}(\mathbf{A})$ such that $x \mapsto \phi_1(x)$ is continuous and satisfies $|\phi_1(x)|_N = 1$.
%
%
%
%
The following auxiliary result, extracted from the proof of \cite[Lemma~4.3]{lam_ays_2016}, will be useful in our analysis.

\begin{lemma}\label{lemma7}
Consider ${\bf A_{0}=-A}$. Then for any $\epsilon>0$, there exists $x_0\in \Omega$ and $r>0$ such that $B'(x_0,r)\subset \O$
\begin{align*}
\overline{\lambda}({\bf A_0'})>\max_{\overline{\O}}\overline{\lambda}({\bf A_0}(x))-\epsilon,
\end{align*}
where ${\bf A_0'}=(b_{ij}')_{N\times N}$, $b_{ij}'=\min\limits_{B'(x_0,r)}b_{ij}$, $i,j\in \{1,2,...,N\}$.

\end{lemma}

Now, we prove some convergence results.

\begin{lemma}\label{Lemm8}
For any $f\in C(\overline{\O})$ and $i=1,2$, one has
\begin{align*}
((-d_i\Delta_N)^{s_i}+|\lambda| +1)^{-1}f\rightarrow \dfrac{f}{|\lambda|+1} \text{ in }C(\overline{\O}) \text { as }d_i \rightarrow 0.
\end{align*}
Furthermore, for any $\epsilon>0$, if $\lambda < \min\limits_{x\in\overline{\O}} \overline{\lambda}(\mathbf{A}(x)) - \epsilon$, then $r\left(K_{\lambda,|\lambda|+1}\right) < 1$ for sufficiently small $\max\{d_1,d_2\}$.
\end{lemma}
\begin{proof}
For any $\epsilon>0$, one can find $g\in C_c^\infty(\R)$ and $\partial_{\textbf{n}}g|_{\partial \O}=0$ via the mollifiers method such that
\begin{align*}
f-\dfrac{\epsilon}{2}\leq g \leq f+\dfrac{\epsilon}{2} \text{ on }\overline{\O}.
\end{align*}
One can prove that $g\in C^{s_i}(\overline{\O})$. 
 Following the proof in \cite[pp 41-42]{lam_ays_2016}, we obtain the convergence.
 
Alternatively, by similar arguments as in \cite[p. 442]{lam_ays_2016}, one can verify that
\begin{align*}
K_{\lambda,|\lambda|+1}\phi_1<\dfrac{|\lambda|+1-\epsilon/2}{|\lambda|+1}\phi_1,
\end{align*}
for $\max\{d_1,d_2\}$ small enough. Hence, by estimate of Collatz–Wielandt radius, $r(K_{\lambda,|\lambda|+1})<1$.

\end{proof}

\begin{remark}\label{remark2}
Note that since $r(K_{\lambda_p,\beta})=1$ as long as $\lambda_p>-\beta$ and $\max\{d_1,d_2\}$ small enough, we have that
\begin{align*}
\lambda_p(d_1,d_2)\geq -\max\limits_{x\in\overline{\O}}\overline{\lambda}({\bf A_{\bf 0}}(x))-\epsilon.
\end{align*}
Thus,
\begin{align*}
\liminf_{\max\{d_1,d_2\}\rightarrow 0}\lambda_p(d_1,d_2)\geq -\max\limits_{x\in\overline{\O}}\overline{\lambda}({{\bf A}_{\bf 0}}(x))=\min\limits_{x\in\overline{\O}}\overline{\lambda}({{\bf A}}(x)).
\end{align*}
\end{remark}

Now, we prove some comparison results before investigating the eigenvalue properties.

\begin{lemma}\label{lemma:com}
Consider $\mathbf{A}'=(a_{ij}'),~\mathbf{A}=(a_{ij})$ in $[C^{0,\alpha}(\overline{\O}) \cap C^{1,0}(\overline{\O})]^{2\times 2}$ for $\alpha$ as in \ref{cond:cond1}, with $a_{12}=a_{21}<0$, $a_{12}'=a_{21}'<0$ and $\partial_{\mathbf{n}} a_{ij}|_{\partial \O} = \partial_{\mathbf{n}} a_{ij}'|_{\partial \O} = 0$ for $i,j=1,2$. Let $\O_0\subset \O$ be a subdomain with smooth boundary. Define $\mathbf{A}_0':=-\mathbf{A}'=(a_{0,ij}')$ and $\mathbf{A}_0:=-\mathbf{A}=(a_{0,ij})$. Assume that $a_{0,ij}' \geq a_{0,ij} \geq 0$. Then, one has
\begin{align*}
\lambda^D_p({\bf A_0},\O_0)\geq \lambda^D_p({\bf A_0'},\O_0)\geq \lambda_p^D({\bf A_0'},\O)\geq \lambda_p({\bf A_0'},\O)=\lambda_p.
\end{align*}
Here, we consider the eigenvalues in Theorem \ref{eigen_theo} as follows $\lambda^D_p({\bf A_0},\O)=\lambda^D_p({\bf A},\O)=\lambda^D_p$, $\lambda_p({\bf A_0},\O)=\lambda_p({\bf A},\O)=\lambda_p$, depending on domain $\O$ and matrix ${\bf A}$.

\end{lemma}

\begin{proof}
We define $(-d\Delta_{0,D})^s,(-d\Delta_{D})^s$ are the fractional Laplacians on $\O_0,~\O$, respectively. Define $K^{0,D}_{\lambda,\beta}$, $K^{D}_{\lambda,\beta}$ are the operator associated with $(-{\bf d}\Delta_{0,D})^{\bf s},(-{\bf d}\Delta_{D})^{\bf s}$, respectively.

First, based on the regularity in \cite[Corollary 2.3 (iv)]{zhao_spatiotemporal_2025}, for any $\textbf{f}\in [C(\overline{\O})]^2,~\textbf{f}\succeq  \textbf{0}$, we get that
\begin{align*}
((-\textbf{d}\Delta_{0,D})^{\textbf{s}}+\beta I)^{-1}\textbf{f}\in [C_0(\overline{\O})]^2;~((-\textbf{d}\Delta_{0,D})^{\textbf{s}}+\beta I)^{-1}\textbf{f}>\textbf{0} \text{ on }\O.
\end{align*}
Following the proof of \cite[Proposition 3.4]{lam_ays_2016}, we  choose $\lambda=\lambda^D_1({\bf A_0},\O_0)$ and $\beta=|\lambda^D_1({\bf A_0},\O_0)|+1$. Then, for any $\textbf{u}\in [C(\overline{\O_0})]^2$, $\textbf{u}\geq \textbf{0} $, it follows from the positivity of $((-\textbf{d}\Delta_D)^{\textbf{s}}+\beta I)^{-1}_0$ that
\begin{align*}
K_{\lambda,\beta}^{0,D}\textbf{u}:=((-\textbf{d}\Delta_{0,D})^{\textbf{s}}+\beta I)^{-1}[{\bf A_0u}+(\lambda+\beta)\textbf{u}]\leq ((-\textbf{d}\Delta_{0,D})^{\textbf{s}}+\beta I)^{-1}[{\bf A_0'u}+(\lambda+\beta)\textbf{u}]=:K^{',0,D}_{\lambda,\beta}\textbf{u}
\end{align*}
Thus, $r\left(K_{\lambda,\beta}^{0,D}\right)\leq r\left(K^{',0,D}_{\lambda,\beta}\right)$. On the other hand, from [Theorem \ref{eigen_theo}, Step 1], one has $r\left(K_{\lambda,\beta}^{0,D}\right)=1$, which leads to $\lambda^D_1({\bf A_0},\O_0)\geq \lambda^D_1({\bf A_0'},\O_0)$.

Next, for simplicity, the comparison is carried out with ${\bf A}_0$ for the two domains $\Omega$ and $\Omega_0$, instead of ${\bf A}_0'$. Recall the formula in \cite[Theorem 3.1, Step 1]{zhao_spatiotemporal_2025} and \cite[Proposition 5.3.2]{Martínez_theory_2001} as follows
\begin{align}\label{formula12}
\left(\omega I +(-d\Delta_D)^s\right)^{-1}=\dfrac{\sin(\pi s)}{\pi}\int_0^\infty\dfrac{\tau^s}{d^s\tau^{2s}+2\omega \tau^s \cos(\pi s)+d^s\omega^s}\left(\tau I -d\Delta_D\right)^{-1}d \tau.
\end{align}

 Based on \eqref{formula12}, for any $f\in C(\overline{\O}),~f\geq 0,~f\not\equiv 0$, one can check that
\begin{align*}
\left((-d\Delta_D)^s+\beta I\right)^{-1}f\geq \left((-d\Delta_{0,D})^s+\beta I\right)^{-1}[f|_{\Omega_0}],
\end{align*}
since the inside of the integral is positive. This implies that
\begin{align*}
((-\textbf{d}\Delta_D)^{\textbf{s}}+\beta I)^{-1}\textbf{f}\geq ((-\textbf{d}\Delta_{0,D})^{\textbf{s}}+\beta I)^{-1}[\textbf{f}|_{\Omega_0}],
\end{align*}
for any $\textbf{f}\in [C(\overline{\O})]^2,~\textbf{f}\succeq  \textbf{0}$.  Hence, if $\textbf{f}\in [C(\overline{\O})]^2,~\textbf{f}\geq \textbf{0},~f_i\not\equiv 0$ on $\O_0$ and $f_i\equiv 0$ on $\O\setminus \O_0$, one has
\begin{align*}
K^D_{\lambda,\beta}[\textbf{f}]\geq  	K^{0,D}_{\lambda,\beta}[\textbf{f}|_{\O_0}],
\end{align*}
where $K_{\lambda,\beta}^{0,D}\textbf{u}:=((-\textbf{d}\Delta_{0,D})^{\textbf{s}}+\beta I)^{-1}[{\bf A_0u}+(\lambda+\beta)\textbf{u}]$. Thus, we conclude that
\begin{align*}
r(K^D_{\lambda,\beta})\geq r(K^{0,D}_{\lambda,\beta}).
\end{align*}
By choosing $\lambda=\lambda^D_p({\bf A_0},\O_0)$, $\beta = |\lambda^D_p({\bf A_0},\O_0)|+1$, [Theorem \ref{eigen_theo}, Step 1] implies  that the desired result.

The rest follows similarly, thanks to the fact that
\begin{align*}
((-\textbf{d}\Delta_D)^{\textbf{s}}+\beta I)^{-1}\textbf{f}>\textbf{0} \text{ on }\O,\\
((-\textbf{d}\Delta_N)^{\textbf{s}}+\beta I)^{-1}\textbf{f}>\textbf{0} \text{ on }\overline{\O}.
\end{align*}
for any $\textbf{f}\in [C(\overline{\O})]^2,~\textbf{f}\succeq  \textbf{0}$ (see \cite[Theorem 3.1, Step 3, Case 1]{zhao_spatiotemporal_2025}).


\end{proof}

\begin{remark}\label{strict}
Note that if we further assume that $\O_0\subset \subset \O$, it then follows classical maximum principle for Dirichlet Laplacian $-d\Delta_D$ that
\begin{align*}
\left(\tau I -d\Delta_D\right)^{-1}[f]>\left(\tau I -d\Delta_{0,D}\right)^{-1}[f|_{\O_0}],~\forall f\in C(\overline{\O}),~f\geq 0,~f\not\equiv 0.
\end{align*}
Thanks to \eqref{formula12}, one has
\begin{align*}
K^D_{\lambda,\beta}[\textbf{f}]> 	K^{0,D}_{\lambda,\beta}[\textbf{f}|_{\O_0}]
\end{align*}
for any $\textbf{f}\in [C(\overline{\O})]^2,~\textbf{f}\succeq {\bf 0},~f_i\not\equiv 0$ on $\O_0$ and $f_i\equiv 0$ on $\O\setminus \O_0$. This implies $r\left(K^D_{\lambda,\beta}\right)>r\left(K^{0,D}_{\lambda,\beta}\right)$, which leads to $\lambda^D_p({\bf A_0},\O_0)>\lambda_p^D({\bf A_0},\O)$.

\end{remark}

We investigate the dependence of the principal eigenvalue $\lambda_p=\lambda_p({\bf d})=\lambda_p(d_1,d_2)$ and $\lambda_p^D=\lambda_p^D({\bf d})=\lambda_p^D(d_1,d_2)$ on the diffusion parameters.

%
%
%
%

\begin{proof} [\bf Proof of Theorem \ref{deffect}]

\ref{monotone}. We follow the approach of \cite{nguyen_dynamics_2022,ninh_vo_asymptotic_2025}.   Without loss of generality, we consider $(d_1,d_2)$ and $(d'_1,d_2)$ for some constants $d_1,d_1',d_2>0$ and assume that $d_1<d_1'$.

Recall that, with $\textbf{u}=(u,v)\in \mathcal{H}^{\bf s}\setminus\{\textbf{0}\}$,
\begin{align*}
\mathcal{J}({ d_1,d_2})(\textbf{u})=b_0(\textbf{u},\textbf{u})=d_1^{s_1}\left<(-\Delta_N)^{s_1}u,u\right>+d_2^{s_2}\left<(-\Delta_N)^{s_2}v,v\right>+\int_{\O}\textbf{u}^T\textbf{A}\textbf{u}dx,
\end{align*}
and
$
\lambda_p(d_1,d_2)=\inf\left\{ \mathcal{J}({ d_1,d_2})(\textbf{u}):\textbf{u}\in \mathcal{H}^{\textbf{s}},~\|\textbf{u}\|_{\mathcal{X}}=1\right\}.
$ 
Then, one has
\begin{align*}
\mathcal{J}({ d_1,d_2})\leq \mathcal{J}({ d_1',d_2}) \text{ on }\mathcal{H}^{\bf s}\setminus\{\textbf{0}\}.
\end{align*}
It implies that
\begin{align*}
\lambda_p(d_1,d_2)\leq \mathcal{J}({ d_1',d_2})(\textbf{u}),
\end{align*}
for any $\textbf{u}\in \{\textbf{v}:\textbf{v}\in \mathcal{H}^{\textbf{s}},~\|\textbf{v}\|_{\mathcal{X}}=1\}$. Choose $\textbf{u}=\varphi_1$, the eigenfunction associated with $\lambda_p(d'_1,d_2)$, we obtain the desired result. 

Furthermore, suppose \ref{cond:cond31} holds for $i=1$,  we consider ${\bf u}=(u,v)$ with $u$ is non-constant. It then follows that
\begin{align*}
\mathcal{J}({ d_1,d_2})< \mathcal{J}({ d_1',d_2}) \text{ implies that }\lambda_p(d_1,d_2)<\lambda_p(d_1',d_2).
\end{align*}

On the other hand, one can check that $\textbf{d}\mapsto \lambda_p(\textbf{d}) \text{ is concave}$
since $\textbf{d} \mapsto \mathcal{J}(\textbf{d})(\textbf{u})$ is concave for any $\textbf{u}\in \{\textbf{v}:\textbf{v}\in \mathcal{H}^{\textbf{s}},~\|\textbf{v}\|_{\mathcal{X}}=1\}$. It is well-known that every concave real function on $(a,b)\times(c,d)\subset \R^2$ is continuous. Thus, it follows  that
\begin{align*}
\textbf{d}\mapsto \lambda_p(\textbf{d}) \text{ is continuous.}
\end{align*}
Based on in \cite[Chapter 7, Theorem 1.8 and Section 2, p 375]{kato_1976}, one can further conclude that $\textbf{d}\mapsto \lambda_p(\textbf{d}) \text{ is analytic.}$


\ref{derivative}. Based on \cite[Chapter 2, Theorem 1.8]{kato_1976} associated with the Riesz projection $P_h({\bf d})$ and the simplicity of $\lambda_p$, one can choose $\{\varphi_1^{\bf d}\}_{d\in (0,\infty)}$ is the set of the positive eigenfunction such that
\begin{align*}
{\bf d}\mapsto \varphi_1^{\bf d} \text{ is analytic in }\mathcal{H}^{\bf s}.
\end{align*}
For any $\boldsymbol{\phi}=(\phi_1,\phi_2)\in \mathcal{H}^{\bf s}$, recall that
\begin{align*}
\left<(-\textbf{d}\Delta_N)^{\textbf{s}}\varphi _1^{\bf d},\boldsymbol{\phi}\right>+ \left<\textbf{A}\varphi_1^{\bf d},\boldsymbol{\phi}\right>=\lambda_p({\bf d})\left<\varphi_1^{\bf d},\boldsymbol{\phi}\right>,
\end{align*}
or
\begin{align*}
d^{s_1}_1\left<\varphi _{1,1}^{\bf d},(-\Delta_N)^{s_1}\phi_1\right>+d^{s_2}_2\left<\varphi_{1,2} ^{\bf d},(-\Delta_N)^{s_2}\phi_2\right>+ \left<\textbf{A}\varphi_1^{\bf d},\phi\right>=\lambda_p({\bf d})\left<\varphi_1^{\bf d},{\bf \phi}\right>.
\end{align*}
Taking the partial derivative with respect to $d_1$ on both sides, we obtain
\begin{align*}
s_1 d_1^{s_1-1}\left<\varphi _{1,1}^{\bf d},(-\Delta_N)^{s_1}\phi_1\right>+d_1^{s_1}\left<\partial_{d_1}\varphi _{1,1}^{\bf d},(-\Delta_N)^{s_1}\phi_1\right>&+d^{s_2}_2\left<\partial_{d_1}\varphi_{1,2} ^{\bf d},(-\Delta_N)^{s_2}\phi_2\right>+\left<\partial_{d_1}\varphi _1^{\bf d},\textbf{A}\boldsymbol{\phi}\right>\\
&=\partial_{d_1}\lambda_p({\bf d})\left<\varphi_1^{\bf d},\boldsymbol{\phi}\right>+\lambda_p({\bf d})\left<\partial_{d_1}\varphi_1^{\bf d},\boldsymbol{\phi}\right>.
\end{align*}
This leads to
\begin{align*}
\left<\partial_{d_1}\varphi _{1,1}^{\bf d},\left[(-\textbf{d}\Delta_N)^{\textbf{s}}+{\bf A}-\lambda_p({\bf d})\right]\boldsymbol{\phi}\right>+s_1 d_1^{-1} \left<\varphi _{1,1}^{\bf d},(-d_1\Delta_N)^{s_1}\phi_1\right>=\partial_{d_1}\lambda_p({\bf d})\left<\varphi_1^{\bf d},\boldsymbol{\phi}\right>.
\end{align*}
Choose $\boldsymbol{\phi}=\varphi_1$, one has
\begin{align*}
\partial_{d_1}\lambda_p({\bf d})\|\varphi _{1}^{\bf d}\|_{\mathcal{X}}^2=s_1 d_1^{-1} \left<\varphi _{1,1}^{\bf d},(-d_1\Delta_N)^{s_1}\varphi_{1,1}^{\bf d}\right>\geq 0.
\end{align*}
Similarly,
\begin{align*}
\partial_{d_2}\lambda_p({\bf d})\|\varphi _{1}^{\bf d}\|_{\mathcal{X}}^2=s_2 d_2^{-1} \left<\varphi _{1,2}^{\bf d},(-d_2\Delta_N)^{s_2}\varphi_{1,2}^{\bf d}\right>\geq 0.
\end{align*}

Now, we can choose $\varphi_1^{\bf d}$ so that $\varphi_{1,i}^{\bf d}$ is non-constant around ${\bf d}^0=(d_1^0,d_2^0)$. Consequently,
\begin{align*}
\partial_{d_i}\lambda_p({\bf d}^0)>0.
\end{align*}
Note that if another holomorphic branch is chosen, the value of $
s_i d_i^{-1} \left\langle \varphi_{1,i}^{\bf d},\, (-d_i\Delta_N)^{s_i}\varphi_{1,i}^{\bf d} \right\rangle$, for $i=1,2$, remains unchanged due to the simplicity of $\lambda_p({\bf d})$.

\ref{lim0}. We only need to prove that
\begin{align*}
\limsup\limits_{\max\{d_1,d_2\}\rightarrow 0} \lambda_p({\bf d})\leq \min\limits_{x\in \overline{\O}}\overline{\lambda}({\bf A}(x)).
\end{align*}

First, we prove the following claim

\textbf{Claim:} Assume $\textbf{A}(x)\equiv\textbf{A}$ constant matrix. Prove that
\begin{align*}
\limsup\limits_{\max\{d_1,d_2\}\rightarrow 0} \lambda_p({\bf d})\leq \overline{\lambda}({\bf A}).
\end{align*}

\begin{proof}

Let $\textbf{u}\in \{\textbf{v}:\textbf{v}\in \mathcal{H}^{\textbf{s}},~\|\textbf{v}\|_{\mathcal{X}}=1\}$. We recall that
\begin{align*}
\lambda_p(d_1,d_2)\leq \mathcal{J}({ d_1,d_2})(\textbf{u}).
\end{align*}
 Passing to the limit $\max\{d_1,d_2\}\rightarrow 0$, we have
\begin{align*}
\limsup_{\max\{d_1,d_2\}\rightarrow 0} \lambda_p(d_1,d_2)&\leq \int_{\O}\textbf{u}^T\textbf{A}\textbf{u}dx.
\end{align*}
Now, let us consider $\phi$ is the eigenfunction corresponding $\overline{\lambda}({\bf A})$ and $\|\phi\|_\mathcal{X}=1$ or $|\phi|_N^2 =|\O|^{-1}$. It is known that $\mathcal{H}^{\bf s}$ is dense in $\mathcal{X}$. As the results, one has
\begin{align*}
\limsup_{\max\{d_1,d_2\}\rightarrow 0} \lambda_p(d_1,d_2)&\leq \int_{\O}\phi^T\textbf{A}\phi dx=\int_{\O}\overline{\lambda}(\textbf{A})|\phi|^2_N dx= \overline{\lambda}(\textbf{A}).
\end{align*}
This completes the claim.
\end{proof}
This claim remains valid under the replacement of $(-\Delta_N)^s$ by the Dirichlet fractional Laplacian $(-\Delta_D)^s$.

\phantom{1}

Now, consider $\textbf{A}_0=-\textbf{A}$ and $\textbf{A}_0'$ in Lemma \ref{lemma7}. Then, from Lemma \ref{lemma:com}, the following chain holds
\begin{align*}
\lambda^D_p({\bf A'_0},B'(x_0,r))\geq \lambda^D_p({\bf A_0},B'(x_0,r))\geq \lambda_p^D({\bf A_0},\O)\geq \lambda_p({\bf A_0},\O)=\lambda_p(d_1,d_2).
\end{align*}
As the results, one has
\begin{align*}
\limsup_{\max\{d_1,d_2\}\rightarrow 0} \lambda_p(d_1,d_2)\leq \limsup_{\max\{d_1,d_2\}\rightarrow 0} \lambda^D_p({\bf A'_0},B'(x_0,r))\leq -\overline{\lambda}({\bf A_0'})\leq -\max_{\overline{\O}}\overline{\lambda}({\bf A_0}(x))+\epsilon,
\end{align*}
for any $\epsilon >0$. This implies that
\begin{align*}
\limsup_{\max\{d_1,d_2\}\rightarrow 0} \lambda_p(d_1,d_2)\leq \limsup_{\max\{d_1,d_2\}\rightarrow 0} \lambda^D_p(d_1,d_2)\leq \min_{\overline{\O}}\overline{\lambda}({\bf A}(x)).
\end{align*}
Similarly, one has
\begin{align*}
\liminf_{\max\{d_1,d_2\}\rightarrow 0}\lambda_p^D(d_1,d_2)\geq \liminf_{\max\{d_1,d_2\}\rightarrow 0}\lambda_p(d_1,d_2)\geq \min\limits_{x\in\overline{\O}}\overline{\lambda}({{\bf A}}(x)).
\end{align*}
We obtain the desired results.


\ref{liminf} Without loss of generality, we assume that $d:=\min\{d_1,d_2\}>1$. Recall
\begin{align*}
\lambda_p(d_1,d_2)&=\inf\left\{ d^{s_1}_1\left<(-\Delta_N)^{s_1}u,u\right>+d^{s_2}_2\left<(-\Delta_N)^{s_2}v,v\right>+\int_{\O}\textbf{u}^T\textbf{A}\textbf{u}dx:\textbf{u}=(u,v)\in \mathcal{H}^{\textbf{s}},~\|\textbf{u}\|_{\mathcal{X}}=1\right\}\\
&=d^{s_1}_1\left<(-\Delta_N)^{s_1}\varphi_{1,1}^d,\varphi_{1,1}^d\right>+d^{s_2}_2\left<(-\Delta_N)^{s_2}\varphi_{1,2}^d,\varphi_{1,2}^d\right>+\int_{\O}(\varphi_{1}^d)^T\textbf{A}\varphi_{1}^ddx,
\end{align*}
where $\varphi_1^d=\left(\varphi_{1,1}^d,\varphi_{1,2}^d\right)$ is the eigenfunction associated with $\lambda_p(d_1,d_2)$, $d$ and $\left\|\varphi_1^d\right\|_{\mathcal{X}}=1$.

Now, choose ${\bf C}=(C_1,C_2)>{\bf 0}$ such that $C^2_1+C_2^2=1$ and ${\bf u}\equiv |\O|^{-1/2}(C_1,C_2)$, one has
\begin{align*}
\lambda_p(d_1,d_2)\leq {\bf u}^T\int_{\O}\textbf{A}dx{\bf u}={\bf C}^T\overline{\bf A}{\bf C},
\end{align*}
which means that $\lambda_p(d)$ is bounded as $d\rightarrow \infty$.
On the other hand, one can check that
\begin{align*}
\left\|\varphi^d_1\right\|_{\mathcal{H}^{\bf s}}^2\leq L\left(\int_{\O}\left|(\varphi_{1}^d)^T(\textbf{A}+I)\varphi_{1}^d \right|dx+\lambda_p(d_1,d_2)\right)\leq L\|{\bf A}+I\|_{\infty} + M
\end{align*}
for some $L,M>0 $ large enough. Thus, it is standard that there exists $\Psi=(\psi_1,\psi_2)\in \mathcal{H}^{\bf s}$ (up-to-subsequence). 
\begin{align*}
\left\{\begin{array}{llll}
\varphi^{d}_1	\rightharpoonup \Psi \text{ on }\mathcal{H}^{\bf s},\\
\varphi^{d}_1	\rightarrow \Psi \text{ on }\mathcal{X},\\
\varphi^{d}_1	\rightarrow \Psi\text{ on a.e. }\overline{\O}.
\end{array}
\right. \text{ as }d\rightarrow \infty,
\end{align*} 
and $\|\Psi\|_{\mathcal{X}}=1$. Alternatively, one has
\begin{align}\label{eq:eigenpro1}
d^{s_1}_1\left<(-\Delta_N)^{s_1}\varphi_{1,1}^d,\phi_1\right>+d^{s_2}_2\left<(-\Delta_N)^{s_2}\varphi_{1,2}^d,\phi_2\right>+\left<{\bf A}\varphi_{1}^d,\boldsymbol{\phi}\right>=\lambda_p(d_1,d_2)\left<\varphi_{1}^d,\boldsymbol{\phi}\right>,
\end{align}
for any $\boldsymbol{\phi}=(\phi_1,\phi_2)\in \mathcal{H}^{\bf s}$. We choose $\boldsymbol{\phi}=(\phi_1,0)$, one has
\begin{align*}
d^{s_1}_1\left<(-\Delta_N)^{s_1}\varphi_{1,1}^d,\phi_1\right>+\int_{\O}\left(a_{11} \varphi_{1,1}^d\phi_1+ a_{22} \varphi_{1,2}^d\phi_1\right) dx =\lambda_p(d_1,d_2)\left<\varphi_{1,1}^d,\phi_1\right>.
\end{align*}
Dividing both sides by $d_1^{s_1}$ and letting $d \to \infty$ (thus $d_1 \to \infty$), one obtains
\begin{align*}
\left<(-\Delta_N)^{s_1}\psi_1,\phi_1\right>=0.
\end{align*}
This leads $\psi_1\equiv C_1^0$ for some constant $C_1^0\geq 0$. Similar to $\psi_2\equiv C_2^0\geq 0$. Based on these results, we define ${\bf C}^0=(C_1^0,C_2^0)\geq (0,0)$. It is known that
\begin{align*}
 \lambda_p(d_1,d_2)\geq \int_{\O}(\varphi_{1}^d)^T\textbf{A}\varphi_{1}^ddx.
\end{align*}
Applying Fatou's lemma, one gets
\begin{align*}
\liminf_{d\rightarrow \infty}\lambda_p(d_1,d_2)\geq ({\bf C}^0)^T\int_{\O}{\bf A}(x)dx{\bf C}^0.
\end{align*}
One can check that $|{\bf C}^0|_N^2=(C_1^0)^2+(C_2^0)^2=|\O|^{-1}$. Recall that
\begin{align*}
\overline{\bf A}=\dfrac{1}{|\O|}\int_{\O}{\bf A}(x)dx,
\end{align*}
and define ${\bf C_0}:=|\O|^{1/2}{\bf C}^0\geq (0,0)$, we get
\begin{align*}
\liminf_{d\rightarrow \infty}\lambda_p(d_1,d_2)\geq {\bf C_0}^T\overline{\bf A}{\bf C_0} \text{ leads to } \lim_{d\rightarrow \infty}\lambda_p(d_1,d_2)={\bf C_0}^T\overline{\bf A}{\bf C_0}.
\end{align*}

By choosing the test functions $\boldsymbol{\phi}=(1,0)$ and $\boldsymbol{\phi}=(0,1)$ in \eqref{eq:eigenpro1} and letting $d\to\infty$, we obtain
\begin{align*}
\overline{\bf A}{\bf C}^0 = ({\bf C_0}^T\overline{\bf A}{\bf C_0}) {\bf C}^0.
\end{align*}

Now, if $({\bf C_0}^T\overline{\bf A}{\bf C_0})\neq \overline{\lambda}(\overline{\bf A})$, direct calculations yield that
\begin{align*}
{\bf V }^T {\bf C}=0,
\end{align*}
where ${\bf V }>\textbf{0}$ is the principal eigenfunction of $\overline{\bf A}$. This leads to ${\bf C}^0=(0,0)$. Contradiction. Hence, we get
\begin{align*}
\lim_{d\rightarrow \infty}\lambda_p({\bf d})=\overline{\lambda}(\overline{\bf A}).
\end{align*}



%

Alternatively, in Dirichlet case, for some $H>0$ large enough, one can prove that
\begin{align*}
\lambda_p^D(d_1,d_2) + 	H
&\geq d^{s_1}\left<(-\Delta_D)^{s_1}\varphi_{1,1}^{d,D},\varphi_{1,1}^{d,D}\right>+d^{s_2}\left<(-\Delta_D)^{s_2}\varphi_{1,2}^{d,D},\varphi_{1,2}^{d,D}\right>\\
&\geq \left(d^{s_1}+d^{s_2}\right)\min\left\{\lambda((-\Delta_D)^{s_1}),\lambda((-\Delta_D)^{s_2})\right\},
\end{align*}
where $\varphi_{1}^{d,D}=\left(\varphi_{1,1}^{d,D},\varphi_{1,2}^{d,D}\right)\in \mathbb{H}^{s_1}\times\mathbb{H}^{s_2}$ is the principal eigenfunction of $(-{\bf d}\Delta_D)^{\bf s}+{\bf A}$, $\left\|\varphi_{1}^{d,D}\right\|_{\mathcal{X}}=1$ , $\lambda((-\Delta_D)^{s_1})=\mu_{D,1}^{s_i}>0$ is the principal eigenvalue of $(-\Delta_D)^{s_i}$ (see \cite[Theorem 3.1]{zhao_spatiotemporal_2025}). Thus, one has
\begin{align*}
\lim\limits_{\min\{d_1,d_2\}\rightarrow \infty} \lambda_p^D({\bf d})=\infty.
\end{align*}
This concludes the proof.

\end{proof}

Next, the dependence of the principal eigenvalue $\lambda_p=\lambda_p({\bf s})=\lambda_p(s_1,s_2)$ on the fractional order is studied.

\begin{proof}[\bf Proof of Theorem \ref{fracorder1}]

\ref{analytic}. For $0<s<1$, we recall the estimate in \cite[Theorem 3.1, Step 4]{zhao_spatiotemporal_2025} as follows.
\begin{align*}
\partial^k_s(\omega I+ (-d\Delta_N)^s)^{-1} \text{ exists }~\|\partial^n_s(\omega I+ (-d\Delta_N)^s)^{-1} \|_{\mathcal{L}(X)}\leq C_{k,\theta,s,d} B\left(\dfrac{s-\theta}{2s},\dfrac{s-\theta}{s}\right),
\end{align*}
for $\omega$ large enough, where $X=L^p(\O)$ or $=C(\overline{\O})$, $p\in (1,\infty)$, $\partial^k_s$ is the $k$-times derivative with respect to $s$, $0<\theta<s$, $C_{k,\theta,s,d}>0$ is a constant depending on $k,\theta,s,d$, $B$ is beta function. It then follows from \cite[Theorem 1.3.2]{kato_1976} that $s\mapsto (-\Delta_N)^s$ is analytic in the sense of Kato. The expression
\begin{align*}
(-\textbf{d}\Delta_N)^{\textbf{s}}=(-d_1\Delta_N)^{s_1}\begin{pmatrix}
1&0\\
0&0
\end{pmatrix}
+ (-d_2\Delta_N)^{s_2}\begin{pmatrix}
0&0\\
0&1
\end{pmatrix}
\end{align*}
shows that $\textbf{s}\mapsto(-\textbf{d}\Delta_N)^{\textbf{s}} $ is also analytic in the sense of Kato. It then follows from the fact $\textbf{A}$ is bounded-analytic in \textbf{s} that $(-\textbf{d}\Delta_N)^{\textbf{s}}+ \textbf{A}$ is analytic in $\textbf{s}$. Based on in \cite[Chapter 7, Theorem 1.8]{kato_1976}, one can conclude that $\textbf{s}\mapsto \lambda_p(\textbf{s}) \text{ is analytic.}$

\ref{limit1s1}. Before proving the results, let us consider the following result, which is consistent with the one in \cite[Proposition 4.4]{di_hitchhiker_2012}.

\phantom{1}

\textbf{Claim:}
Consider  $X=H^2(\O)\cap H^1_0(\O)$ if $B=D$ and $X=\{v\in H^2(\O):\partial_{\bf n}v=0\}$ if $B=N$. Let  $u\in X$. Then, one has
\begin{align*}
(-\Delta_B)^s u \rightarrow -\Delta_B u \text{ on }L^2(\O) \text{ as }s\rightarrow 1^-,
\end{align*}
where $B$ is either $D$ or $N$. Furthermore, one has
\begin{align*}
(-\Delta_N)^s u \rightarrow I-P_0 \text{ on }L^2(\O) \text{ as }s\rightarrow 0^+,\quad~(-\Delta_D)^s u \rightarrow I \text{ on }L^2(\O) \text{ as }s\rightarrow 0^+,
\end{align*}
where $P_0u=\dfrac{1}{|\O|}\displaystyle\int_\O udx$.

\begin{proof}[Proof of claim]
Without loss of generality, one can assume that $s>\dfrac{1}{2}$. Due to Parseval's identity, it is easy to check that
\begin{align*}
\|(-\Delta_B)^s u -(-\Delta_B) u\|_{L^2(\O)}^2 = \sum_{k=0}^\infty\left|\mu_{B,k}^s-\mu_{B,k}\right|^2u_k^2.
\end{align*}
Thanks to the fact that $\mu_{B,k}^s\leq \mu_{B,k}$ for all $k$ large enough, we can apply Lebesgue dominated convergence theorem to obtain the desired result.

The last two results follow by similar arguments.

\end{proof}

Returning to the main result, we define $s_{\min} = \min\{s_1, s_2\}$ and, thus, $\lambda_p(s_{\min})=\lambda_p({\bf s})$. Without lost of generality, we assume that $s_{\min}>3/4$. Recall the existence of $\lambda_1\left(-\textbf{d}\Delta_N+ {\bf A}\right)$ from the work of \cite[Lemma 3.2]{lam_ays_2016}. Similar arguments to the one in Step 3, Theorem \ref{eigen_theo}, one can prove that
\begin{align}\label{minlap}
\lambda_1\left(-\textbf{d}\Delta_N+ {\bf A}\right)=\inf\left\{\left<-\textbf{d}\Delta_N {\bf u},{\bf u}\right>+\left<{\bf Au},{\bf u}\right>:\textbf{u}\in [H^1(\O)]^2,~\|\textbf{u}\|_{\mathcal{X}}=1  \right\}.
\end{align}
Consider ${\bf v}\in [H^2(\O)]^2$ satisfying $\partial_{\bf n}{\bf v}=0,~\|{\bf v}\|_{\mathcal{X}}=1$ and ${\bf s}_n=(s_{1,n},s_{2,n})$ such that $s_{\min,n}\rightarrow 1^-$ as $n\rightarrow \infty$. One can check that
\begin{align*}
\lambda_p(s_{\min,n})\leq \left<(-\textbf{d}\Delta_N)^{\textbf{s}_n}{\bf v},{\bf v}\right> + \left<{\bf Av},{\bf v}\right>.
\end{align*}
By applying Claim and passing to the limit $n \to \infty$, one has
\begin{align*}
\limsup_{n\rightarrow \infty}\lambda_p(s_{\min,n}) \leq  \left<(-\textbf{d}\Delta_N){\bf v},{\bf v}\right> + \left<{\bf Av},{\bf v}\right>.
\end{align*}
Thanks to the denseness of $D(-\Delta_N)$ in $H^1(\O)$, one has
\begin{align*}
\limsup_{n\rightarrow \infty}\lambda_p(s_{\min,n}) \leq  \lambda_1\left(-\textbf{d}\Delta_N+ {\bf A}\right).
\end{align*}

Next, consider $\varphi_1^n$ is the principal eigenfunction of $\lambda_p(s_{\min,n})$ such that $\left\|\varphi_1^n\right\|_{\mathcal{X}}=1$. Then, one has
\begin{align*}
\left<(-\textbf{d}\Delta_N)^{\textbf{s}_n}\varphi_1^n,\varphi_1^n\right> = \lambda_p(s_{\min,n}) - \left<{\bf A}\varphi_1^n,\varphi_1^n\right>\leq M + \left\| {\bf A}\right\|_{\infty} \left\|\varphi_1^n\right\|^2_{\mathcal{X}},
\end{align*}
for some $M>0$ large enough. Now, let $N_0\in \N$ so that $\mu_{N,k}>1$, the eigenvalue of $-\Delta_N$, for any $k>N_0$ and $\mu_{N,k}\leq 1$ for any $k\leq N_0$. Then, one has
\begin{align*}
\left<(-\Delta_N)^{3/4}\varphi_{1,1}^n,\varphi_{1,1}^n\right>&=\sum_{k=0}^{\infty}\mu_{N,k}^{3/4}\left<\varphi_{1,1}^n,\phi_k\right>^2\\
&=\sum_{i=N_0+1}^{\infty}\mu_{N,k}^{3/4}\left<\varphi_{1,1}^n,\phi_k\right>^2+\sum_{k=0}^{N_0}\mu_{N,k}^{3/4}\left<\varphi_{1,1}^n,\phi_k\right>^2\\
&\leq \sum_{k=0}^{\infty}\mu_{N,k}^{s_{1,n}}\left<\varphi_{1,1}^n,\phi_k\right>^2+\sum_{k=0}^{\infty}\left<\varphi_{1,1}^n,\phi_k\right>^2\\
&\leq C(d_1)\left<(-d_1\Delta_N)^{s_{1,n}}\varphi_{1,1}^n,\varphi_{1,1}^n\right>+1,
\end{align*}
for some constant $C_1(d_1)>0$. Similar to $\varphi_{1,2}$. Hence, one has $\left\{\varphi_1^n\right\}_{n\in \N}$ is bounded on $[H^{3/4}(\O)]^2$. 
Thus, it is standard that there exists $\Psi=(\psi_1,\psi_2)\in [H^{3/4}(\O)]^2$ (up-to-subsequence).
\begin{align*}
\left\{\begin{array}{llll}
\varphi^{n}_1	\rightharpoonup \Psi \text{ on }[H^{3/4}(\O)]^2,\\
\varphi^{n}_1	\rightarrow \Psi \text{ on }\mathcal{X},\\
\varphi^{n}_1	\rightarrow \Psi\text{ on a.e. }\overline{\O},
\end{array}
\right. \text{ as }n\rightarrow \infty.
\end{align*} 
Now, we use the eigen-expansion as follows
\begin{align*}
\left<(-\textbf{d}\Delta_N)^{\textbf{s}_n}\varphi_1^n,\varphi_1^n\right>=\left(\sum_{k=0}^\infty d_1^{s_{1,n}}\mu_{N,k}^{s_{1,n}}\left<\varphi_{1,1}^n,\phi_k\right>^2,\sum_{k=0}^\infty d_2^{s_{2,n}}\mu_{N,k}^{s_{2,n}}\left<\varphi_{1,2}^n,\phi_k\right>^2 \right).
\end{align*}
By applying Fatou's lemma, we get that
\begin{align*}
\liminf_{n\rightarrow \infty}\sum_{k=0}^\infty d_1^{s_{1,n}}\mu_{N,k}^{s_{1,n}}\left<\varphi_{1,1}^n,\phi_k\right>^2\geq \sum_{k=0}^\infty \liminf_{n\rightarrow \infty} d_1^{s_{1,n}}\mu_{N,k}^{s_{1,n}}\left<\varphi_{1,1}^n,\phi_k\right>^2= \sum_{k=0}^\infty  d_1\mu_{N,k}\left<\psi_1,\phi_k\right>^2.
\end{align*}
 Here, we use the strong convergence in $\mathcal{X}=[L^2(\O)]^2$. Similar to the case $s_2$. Thus, one get
\begin{align*}
\liminf_{n\rightarrow \infty}\left<(-\textbf{d}\Delta_N)^{\textbf{s}_n}\varphi_1^n,\varphi_1^n\right>\geq \left<(-\textbf{d}\Delta_N)\Psi,\Psi\right>\geq 0.
\end{align*}
This shows that $\Psi \in [H^1(\O)]^2$ and $\|\Psi\|_{\mathcal{X}}=1$. It follows from \eqref{minlap} that
\begin{align*}
\lambda_1\left(-\textbf{d}\Delta_N+ {\bf A}\right)\leq \liminf_{n\rightarrow \infty}\lambda_p(s_{\min,n}).
\end{align*}

\ref{limit1s0} The proof of the case $\max\{s_1,s_2\}\rightarrow 0^+$ is similar. The convergence is the $L^2$-weak along with the application of Fatou's lemma to obtain the result.

We conclude the proof.

\end{proof}


Now, we investigate the effects of domains on the principal eigenvalue for Dirichlet case. For this purpose, consider $\O_i\subset \R^N$ is bounded domain with smooth boundary,  we define
\begin{align*}
\mathcal{K}_{D,i}:=(-\textbf{d}\Delta_{i,D})^{\textbf{s}} + {\bf A}\text{ on }H^{s_1}(\O_i)\times H^{s_2}(\O_i).
\end{align*}
where $(-\textbf{d}\Delta_{i,D})^{\textbf{s}}$ is the fractional Laplacian on $\O_i$. The operator $\mathcal{K}_{D,i}$ depends on the domain $\O_i$. In this section, we consider the principal eigenvalue $\lambda^D_p=\lambda^D_p(\O_i)$, depending on the domain $\O_i$.
We further assume that ${\bf A}(x)\equiv {\bf A}$, a constant matrix. With $\O_l:=l\O$, we consider
\begin{align*}
(-\textbf{d}\Delta_{l,D})^{\textbf{s}}\textbf{u} + {\bf Au}=\lambda^D_p(\O_l) {\bf u}\text{ on }\O_l.
\end{align*}
Since the matrix is constant, this eigenvalue problem can be considered on $\O_l$ for any $l>0$. Thanks to Section \ref{Section 2.2}, there exists a unique $L^2$-normalized  $\varphi_l^D>{\bf 0}$ such that
\begin{align*}
(-\textbf{d}\Delta_{l,D})^{\textbf{s}}\varphi_l^D(y) + {\bf A} \varphi_l^D=\lambda^D_p(\O_l)\varphi_l^D,~y\in\O_l.
\end{align*}

Note that direct calculation yields that
\begin{align*}
-d\Delta u= \lambda u \text{ on $\O$ implies that } -d\Delta v= \dfrac{\lambda}{l^2} v \text{ on }\O_l,
\end{align*}
where $v(x)=u(x/l)$. Thus, one has
\begin{align*}
\mu_{D,k}(\O_l) = \dfrac{1}{l^2} \mu_{D,k}(\O).
\end{align*}
This implies that
\begin{align*}
(-\textbf{d}\Delta_{l,D})^{\textbf{s}}v(y)=\dfrac{1}{l^{2s}} (-\textbf{d}\Delta_{1,D})^{\textbf{s}}u(y/l)\quad \text{ or }\quad(-\textbf{d}\Delta_{1,D})^{\textbf{s}}u(x)=l^{2s}(-\textbf{d}\Delta_{l,D})^{\textbf{s}}v( l x).
\end{align*}
for $y\in \O_l$ and $x\in \O$.

\begin{theorem}
Assume  ${\bf A}(x)\equiv {\bf A}$ constant matrix and $a_{12}=a_{21}<0$. For $l>0$, consider $\O_{l}:=l\O$. Then, one has
\begin{align*}
\lim_{l \rightarrow \infty}\lambda^D_p(\O_l)=\overline{\lambda}({\bf A});\quad \lim_{l \rightarrow 0}\lambda^D_p(\O_l)=\infty.
\end{align*}

\end{theorem}

\begin{proof}
Recall that  
\begin{align*}
\lambda_p^D(\O_{l_2})<\lambda_p^D(\O_{l_1}) \text{ if }l_1<l_2.
\end{align*}
It is easy to check that
\begin{align*}
\left<\mathcal{K}_{D,l}{\bf u},{\bf u}\right>_{\O_l}\geq \left<{\bf Au},{\bf u}\right>_{\O_l} \geq \overline{\lambda}({\bf A}) \left<{\bf u},{\bf u}\right>_{\O_l},~\forall {\bf u}\in H^{s_1}(\O_l)\times H^{s_2}(\O_l),
\end{align*}
which implies that
\begin{align*}
\lambda^D_p(\O_l)\geq \overline{\lambda}({\bf A}).
\end{align*}

On the other hand, recall $(\mu_{D,k},\phi_k)^\infty_{k=0}$ are eigenpairs of Dirichlet Laplacian $-\Delta_D$ on $\O$. Consider ${\bf v}=(v_1,v_2)$, $|v|_N=1$ such that 
\begin{align*}
\text{${\bf Av}=\overline{\lambda}({\bf A}){\bf v}$ and }{\bf u}_k^l:=\phi_k^l {\bf v}=\left(\phi_k^l v_1,\phi_k^l v_2\right).
\end{align*}
where $\phi_k^l$ is the $L^2$-normalized $k$-eigenfunction of Dirichlet Laplacian on $\O_{l}$. Then, one has
\begin{align*}
\left<(-\textbf{d}\Delta_{l,D})^{\textbf{s}}{\bf u}_k^l,{\bf u}_k^l \right>
&=d^{s_1}_1\left<(-\Delta_{l,D})^{s_1}\phi_k^l v_1,\phi_k^l v_1\right>+d^{s_2}_2\left<(-\Delta_{l,D})^{s_2}\phi_k^l v_2,\phi_k^l v_2\right>\\
&=d^{s_1}_1l^{-2s_1}\mu_{D,k}^{s_1}|v_1|^2+d^{s_2}_2l^{-2s_2}\mu_{D,k}^{s_2}|v_2|^2,
\end{align*}
and
\begin{align*}
\left<{\bf A}{\bf u}_k^l,{\bf u}_k^l\right>=\overline{\lambda}({\bf A}).
\end{align*}
As the results, one can check that
\begin{align*}
\lambda^D_p(\O_l)\leq d^{s_1}_1l^{-2s_1}\mu_{D,k}^{s_1}+d^{s_2}_2l^{-2s_2}\mu_{D,k}^{s_2}+\overline{\lambda}({\bf A}).
\end{align*}
This implies that
\begin{align*}
\lim_{l \rightarrow \infty}\lambda^D_p(\O_l)=\overline{\lambda}({\bf A}).
\end{align*}

Next, without loss of generality, we assume that $s_1<s_2$ and $l<1$ since considering $l\rightarrow 0$. Then, for any ${\bf u}=(u_1,u_2)\in H^{s_1}(\O_l)\times H^{s_2}(\O_l)$, one has
\begin{align*}
\left<\mathcal{K}_{D,l}{\bf u},{\bf u}\right>_{\O_l} &\geq \left<(-\textbf{d}\Delta_{l,D})^{\textbf{s}}{\bf u},{\bf u} \right>+ \overline{\lambda}({\bf A}) \left<{\bf u},{\bf u}\right>_{\O_l}\\
&\geq l^{-2s_1}d_1^{s_1}\mu_{D,1}^{s_1}\|u_1\|_{L^2(\O_l)}^2 +  l^{-2s_2}d_2^{s_2}\mu_{D,2}^{s_2}\|u_2\|_{L^2( \O_l)}^2 + \overline{\lambda}({\bf A})\|{\bf u}\|_{[L^2(\O_l)]^2}^2,\\
&\geq l^{-2s_1} \min\left\{d_1^{s_1}\mu_{D,1}^{s_1},d_1^{s_2}\mu_{D,2}^{s_2}\right\}\|{\bf u}\|_{[L^2(\O_l)]^2}^2+ \overline{\lambda}({\bf A})\|{\bf u}\|_{[L^2(\O_l)]^2}^2.
\end{align*}
It follows that
\begin{align*}
\lambda^D_p(\O_l)\geq d_1^{s_1}l^{-2s_1}\mu_{D,1}^{s_1}+\overline{\lambda}({\bf A}).
\end{align*}
This concludes the proof.

\end{proof}

We end this section with some comments on  the asymptotic behavior of the Dirichlet principal eigenfunction $\varphi_1^{(d_1,d_2)}$,   associated with $\lambda_p^D(d_1,d_2)$, as $\max\{d_1,d_2\}\rightarrow 0$. Define $d=\max\{d_1,d_2\}$ and consider 
\begin{align*}
\varphi_1^{d}=\varphi_1^{(d_1,d_2)}>{\bf 0},~\left\|\varphi_1^{d}\right\|_{\mathcal{H}^{\bf s}}=1.
\end{align*}
For every $d_1,d_2>0$, the choice of $\varphi_1^{d}$ is unique. Indeed, if another eigenfunction $\phi_1$ with the same properties exists, the simplicity of the principal eigenvalue implies that $\phi_1 \equiv \varphi_1^{d}$.

Now, let $(d^n_1,d^n_2)$ be a sequence such that $d^n=\max\{d^n_1,d^n_2\}\rightarrow 0$. Then, it is standard  (up-to-subsequence) that
\begin{align*}
\left\{\begin{array}{llll}
\varphi^{d^n}_1	\rightharpoonup \varphi \text{ on }\mathcal{H}^{\bf s},\\
\varphi^{d^n}_1	\rightarrow \varphi \text{ on }\mathcal{X},\\
\varphi^{d^n}_1	\rightarrow \varphi \text{ on a.e. }\overline{\O},
\end{array}
\right.\text{ as }n\rightarrow \infty,
\end{align*}
for some $\varphi\in \mathcal{H}^{\bf s}$ and $\varphi \geq 0$ a.e. on $\overline{\O}$. As the results, for any $\phi\in \mathcal{H}^{\bf}$, we can check that
\begin{align*}
\left<(-\Delta_N)^{\bf s}\varphi^{d^n}_1,\phi\right>=\left<\varphi^{d^n}_1,(-\Delta_N)^{\bf s}\phi\right>\rightarrow \left<\varphi,(-\Delta_N)^{\bf s}\phi\right>=\left<(-\Delta_N)^{\bf s}\varphi,\phi\right>.
\end{align*}
Thus,
\begin{align*}
\left<(-\textbf{d}\Delta_N)^{\textbf{s}}\varphi^{d^n}_1,\phi\right> + \left<\textbf{A}\varphi^{d^n}_1,\phi\right> =\lambda_p(d^n)^D\left<\varphi^{d^n}_1,\phi\right>,
\end{align*}
or
\begin{align*}
d_1^{s_1}\left<(-\Delta_N)^{s_1}\varphi^{d^n}_{1,1},\phi_1\right>+d_2^{s_2}\left<(-\Delta_N)^{s_2}\varphi^{d^n}_{1,2},\phi_2\right>+\int_{\O}(\varphi_1^{d^n})^T\textbf{A}\phi dx=\lambda_p^D(d^n)\int_{\O}(\varphi_1^{d^n})^T\phi dx.
\end{align*}
Passing to the limit $d^n\rightarrow 0$, one has
\begin{align*}
\int_{\O}\varphi^T\textbf{A}\phi dx=\min\limits_{x\in \overline{\O}}\overline{\lambda}({\bf A}(x))\int_{\O}\varphi^T\phi dx,~\forall \phi \in \mathcal{H}^{\bf s}.
\end{align*}
As the results, one has
\begin{align*}
\textbf{A}(x)\varphi(x)=\min\limits_{x\in \overline{\O}}\overline{\lambda}({\bf A}(x))\varphi(x)\text{ a.e. on }\overline{\O}.
\end{align*}
Consider $x_0 \in \overline{\O}$ such that $\min\limits_{x\in \overline{\O}}\overline{\lambda}({\bf A}(x))=\overline{\lambda}({\bf A}(x_0))$, it is possible that $\varphi(x_0)>0$ but not possible in the other cases $x\neq x_0$. Following the result in \cite[Proposition 1.3.17]{lam_Introduction_2022}, we expect that 
\begin{align*}
\varphi(x)\equiv 0 \text{ on }\O_0,
\end{align*}
where $\O_0:=\left\{x\in \O:\min\limits_{x\in \overline{\O}}\overline{\lambda}({\bf A}(x))<\overline{\lambda}({\bf A}(x))\right\}$. We leave this question open for future research.


%
%
%

%
%
%
%

\section{\bf Well-posedness of system \eqref{eq:main}}\label{section:3}
Based on the work of \cite[Section 2]{zhao_spatiotemporal_2023}, $A:=-(-d\Delta_N)^s$ generates an analytic semigroup $e^{tA}$ and 
\begin{align}\label{eq:frac}
e^{-t(-d\Delta_N)^s}w=\int_0^{\infty}T_{s,t}(\tau)e^{\tau d\Delta_N}wd\tau,~\forall w \in C(\overline{\O}).
\end{align}
where $\{T_{s,t}\}_{(s,t)\in (0,1)\times (0,\infty)}\subset L^1((0,\infty))$ is a family of non-negative functions such that
\begin{align*}
\int_0^{\infty}T_{s,t}(\tau)d\tau=1,~t>0;~T_{s,t}* T_{s,\eta}=T_{s,t+\eta},~t,~\eta>0;~T_{s,t}(\tau)=t^{-1/s}T_{s,1}\left(t^{-1/s}\tau\right),~\tau>0,~t>0.
\end{align*}
Thus, one can check that $e^{tA}w\geq 0$ whenever $w\geq 0$ for $w\in L^2(\O)\setminus\{0\}$ or $C(\overline{\O})\setminus\{0\}$ thanks to the strong positivity of $e^{td\Delta_N}w$. It implies that the comparison principle holds for the semigroup $e^{tA}w$ of $-(-d\Delta_N)^s$.

Furthermore, inspired by \cite[Proposition 2.7]{zhao_spatiotemporal_2025}, for any $w\in C(\overline{\O})\setminus\{0\}$, $w\geq 0$, one has
\begin{align*}
e^{td\Delta_N}w >0,~\forall t>0 \text{ on }\overline{\O}.
\end{align*}
For $T_1>0$, there exists $\delta>0$ such that
\begin{align*}
e^{T_1d\Delta_N}w>\delta \text{ on }\overline{\O}.
\end{align*}
Thanks to the parabolic comparison principle, one has
\begin{align*}
e^{(t+T_1)d\Delta_N}w>e^{td\Delta_N}\delta=\delta,~\forall t>0 \text{ on }\overline{\O}.
\end{align*}
Thus, one can check that
\begin{align*}
e^{-t(-d\Delta_N)^s}w&=\int_0^{\infty}T_{s,t}(\tau)e^{\tau d\Delta_N}wd\tau\geq \delta \int_{T_1}^{\infty}T_{s,t}(\tau) d\tau.
\end{align*}
Since $\displaystyle\int_{0}^{\infty}T_{s,t}(\tau) d\tau=1$, we choose $T_1>0$ small enough so that $\displaystyle\int_{T_1}^{\infty}T_{s,t}(\tau) d\tau>0$. This implies that \begin{align*}
e^{-t(-d\Delta_N)^s}w>\epsilon \text{ on }\overline{\O}, \text{ for any }t>0,
\end{align*}
for some $\epsilon>0$ small enough.


\subsection{\bf Local existence and uniqueness}

We consider the local existence of the solution to the system \eqref{eq:main} on $\mathcal{X}$ or $\mathcal{C}$.  The technique is standard. For completeness, we investigate some important details as follows
\begin{enumerate}
\item[•]  Define $e^{-t(-\textbf{d}\Delta_N)^\textbf{s}}\textbf{u}=(e^{-t(-d_1\Delta_N)^{s_1}}u_1,e^{-t(-d_2\Delta_N)^{s_2}}u_2)$ for $\textbf{u}\in \mathfrak{S}:=\mathcal{X} \text{ or }\mathcal{C}$. Based on \cite[Lemma 3.1]{zhao_spatiotemporal_2023}, there exists $C_0>0$ such that
\begin{align*}
\|e^{-t(-\textbf{d}\Delta_N)^\textbf{s}}\textbf{w}\|\leq C_0\|\textbf{w}\|,~\forall\textbf{w}\in \mathfrak{S},~ t>0.
\end{align*}
Without loss of generality, assume that $C_0\geq 1$. Here, $\|\cdot\|$ is the product norm associated with $\mathfrak{S}$.

\item[•] For any bounded $\textbf{u}_0:=(u_0,v_0)\in \mathfrak{S}$, one can check that the mild solution of the system \eqref{eq:main} is as follows
\begin{align*}
\textbf{u}(t)=e^{-t(-\textbf{d}\Delta_N)^\textbf{s}}\textbf{u}_0+\int_0^te^{-(t-\tau)(-\textbf{d}\Delta_N)^\textbf{s}}\left(\mathbf{B}+\mathcal{G}\right)\textbf{u}(\tau)d\tau.
\end{align*}
Consequently, with $T>0$ that will be chosen later, we define
\begin{align*}
\mathcal{F}(\textbf{u})(t):=e^{-t(-\textbf{d}\Delta_N)^\textbf{s}}\textbf{u}_0+\int_0^te^{-(t-\tau)(-\textbf{d}\Delta_N)^\textbf{s}}\left(\mathbf{B}+\mathcal{G}\right)\textbf{u}(\tau)d\tau,~t\in [0,T].
\end{align*}
It suffices to find the fixed point of $\mathcal{F}$.
\item[•] There exists $M>0$ such that
\begin{align*}
\|\textbf{u}_0\|\leq M
\end{align*}
Define
\begin{align*}
&C_1:=\max\{|a_{max}|,|a_{min}|,|b_{max}|,|b_{min}|,H'(0),G'(0)\},\\
& \mathfrak{X}:=\left\{\textbf{w}\in C([0,T],\mathfrak{S}):\|\textbf{w}\|\leq C_0(1+C_1)M\right\}.
\end{align*}
It is easy to see that $C_1$ is the Lipschitz constant of $\left(\mathbf{B}+\mathcal{G}\right)$ and $\textbf{u}_0\in \mathfrak{X}$.
%
Then, one can choose $T>0$ small enough so that 
\begin{align*}
\mathcal{F}(\textbf{u})\in \mathfrak{X},~\forall \textbf{u}\in \mathfrak{X} \text{ and it is a contraction mapping.}
\end{align*}
Consequently, the local existence and uniqueness are ensured.

\end{enumerate}

Furthermore, for any $\textbf{u}_0,\textbf{v}_0\in \mathfrak{S}$, one has
\begin{align*}
\|\textbf{u}(t;\textbf{u}_0)-\textbf{u}(t;\textbf{v}_0)\|\leq C_0\|\textbf{u}_0-\textbf{v}_0\|+C_0K\int_0^t\|\textbf{u}(\tau;\textbf{u}_0)-\textbf{u}(\tau;\textbf{v}_0)\|d\tau,
\end{align*}
for some $K>0$ associated with $C_1$. Then, by Gronwall's inequality, one has
\begin{align}\label{continuous}
\|\textbf{u}(t;\textbf{u}_0)-\textbf{u}(t;\textbf{v}_0)\|\leq  C_0\|\textbf{u}_0-\textbf{v}_0\|e^{C_0Kt},~\forall t>0.
\end{align}

Next, we establish the comparison principle as an essential step toward proving the global-in-time existence.

\subsection{\bf Comparison principle}\label{section:comparison}

We define monotonicity for functions on $\O\times\mathbb{R}^2$ as follows

\begin{definition}\label{def:5}
Let $\textbf{G}:\overline{\O}\times\R^2\rightarrow \R^2$, $\textbf{G}=\textbf{G}(x,\textbf{u})$ be a continuous function. We say $\textbf{G}$ is \textit{monotone} if for any $K_1,~K_2>0$, there exists $\lambda>0$ such that 
\begin{align*}
0\leq \textbf{G}(x,a,\textbf{u})+\lambda \textbf{u}\leq \textbf{G}(x,a,\textbf{v})+\lambda \textbf{v},
\end{align*}
for any $ \textbf{u},~\textbf{v}\in \R^2,~ \textbf{0}\leq \textbf{u}\leq \textbf{v}\leq (K_1,K_1),~x\in \overline{\O}.$
\end{definition}

Let us check the monotonicity for the nonlinear term $\mathbf{B}+\mathcal{G}$. For any $K_1,~K_2>0$, let $\lambda>0$ be chosen later, consider $\textbf{u}=(u_1,u_2)\in \R^2$ and $\textbf{u}\geq \textbf{0}$, one has
\begin{align*}
\mathbf{B}\textbf{u}+\mathcal{G}\textbf{u}+\lambda \textbf{u}=\left(\begin{matrix}
(\lambda-a)u_1 + H(u_2)\\
G(u_1)+(\lambda-b)u_2
\end{matrix}\right)\geq 0,
\end{align*}
provided that $\lambda\geq \max\{a_{\max},b_{\max}\}$. Furthermore, with $\textbf{v}=(v_1,v_2)$ and $\textbf{v}\geq \textbf{u}$, it is easy to check that
\begin{align*}
\left(\begin{matrix}
(\lambda-a)u_1 + H(u_2)\\
G(u_1)+(\lambda-b)u_2
\end{matrix}\right)\leq 
\left(\begin{matrix}
(\lambda-a)v_1 + H(v_2)\\
G(v_1)+(\lambda-b)v_2
\end{matrix}\right),
\end{align*}
since $H,G$ are strictly increasing functions. We conclude the nonlinear term is monotone.

Motivated by \cite[Theorem 4.5]{magal_monotone_2019}, we consider $\textbf{u}_0,\textbf{v}_0\in \mathfrak{S}$ and $\textbf{u}_0\leq \textbf{v}_0$. Similar to the local existence part, we choose $M$ larger so that 
\begin{align*}
\textbf{u}_0,\textbf{v}_0\in \mathfrak{X}.
\end{align*}
And, by choosing $T$ small enough, one can check that 
\begin{align*}
\mathcal{F}_{\textbf{u}_0}^{\lambda}(\textbf{u})(t):=e^{-\lambda t}e^{-t(-\textbf{d}\Delta_N)^\textbf{s}}\textbf{u}_0+\int_0^t e^{-\lambda (t-\tau)}e^{-(t-\tau)(-\textbf{d}\Delta_N)^\textbf{s}}\left(\mathbf{B}+\mathcal{G}+\lambda I\right)\textbf{u}(\tau)d\tau,~t\in [0,T],\\
\mathcal{F}_{\textbf{v}_0}^{\lambda}(\textbf{v})(t):=e^{-\lambda t}e^{-t(-\textbf{d}\Delta_N)^\textbf{s}}\textbf{v}_0+\int_0^te^{-\lambda (t-\tau)}e^{-(t-\tau)(-\textbf{d}\Delta_N)^\textbf{s}}\left(\mathbf{B}+\mathcal{G}+\lambda I\right)\textbf{v}(\tau)d\tau,~t\in [0,T],
\end{align*}
have fixed points $\textbf{u}(t;\textbf{u}_0)$ and $\textbf{v}(t;\textbf{v}_0)$ in $\mathfrak{X}$, respectively. Note that the solution $\textbf{u}$ here coincides with that in the previous section due to uniqueness on the suitable time interval. 

Next, thanks to the positivity of the semigroup and monotonicity, we have that
\begin{align*}
&\mathcal{F}_{\textbf{u}_0}^{\lambda}(\textbf{u}_0)\leq \mathcal{F}^{\lambda}_{\textbf{u}_0}(\textbf{v}_0) \leq \mathcal{F}^{\lambda}_{\textbf{v}_0}(\textbf{v}_0) ,\\
&\mathcal{F}^{\lambda}_{\textbf{u}_0}(\mathcal{F}^{\lambda}_{\textbf{u}_0}(\textbf{u}_0)) \leq \mathcal{F}^{\lambda}_{\textbf{u}_0}(\mathcal{F}^{\lambda}_{\textbf{v}_0}(\textbf{v}_0)) \leq \mathcal{F}^{\lambda}_{\textbf{v}_0}(\mathcal{F}^{\lambda}_{\textbf{v}_0}(\textbf{v}_0)),
\end{align*}
By induction, one gets
\begin{align*}
\mathcal{F}_{\textbf{u}_0}^{\lambda,(n)}(\textbf{u}_0) \leq \mathcal{F}_{\textbf{v}_0}^{\lambda,(n)}(\textbf{v}_0),~\forall n \in \N.
\end{align*}
Passing to the limit $n\rightarrow \infty$, we obtain
\begin{align*}
\textbf{u}(t;\textbf{u}_0)\leq \textbf{v}(t;\textbf{v}_0).
\end{align*}
Thus, the comparison principle holds in the suitable local-time interval.

On the other hand, for $\lambda$ large enough and $\textbf{u}_0\in \mathcal{C}\setminus\{(0,0)\},~\textbf{u}_0\geq \textbf{0}$, it follows that
\begin{align*}
\textbf{u}(t;\textbf{u}_0)=e^{-\lambda t}e^{-t(-\textbf{d}\Delta_N)^\textbf{s}}\textbf{u}_0+\int_0^t e^{-\lambda (t-\tau)}e^{-(t-\tau)(-\textbf{d}\Delta_N)^\textbf{s}}\left(\mathbf{B}+\mathcal{G}+\lambda I\right)\textbf{u}(\tau;\textbf{u}_0)d\tau.
\end{align*}
By applying the comparison principle, it follows that $\textbf{u}(t;\textbf{u}_0)\geq 0$. Thus, there exists $\epsilon>0$ small enough such that
\begin{align*}
\textbf{u}(t;\textbf{u}_0)\geq e^{-\lambda t}e^{-t(-\textbf{d}\Delta_N)^\textbf{s}}\textbf{u}_0\geq (\epsilon,\epsilon) \text{ on }\overline{\O}\text{ for any suitable }t>0.
\end{align*}

In addition, in view of \cite[Proposition 5.1 - 5.4]{magal_monotone_2019}, if $\overline{\bf v}$ is a super-solution of \eqref{eq:main1} and $\textbf{u}_0\leq \overline{\bf v}(0)$, one has
\begin{align*}
\textbf{u}(t;\textbf{u}_0)\leq \overline{\bf v}(t) \text{ on }\overline{\O},~\forall t>0.
\end{align*}
The same conclusion holds for sub-solutions with the reverse inequality.

\subsection{\bf Global-in-time existence}
We first construct a super-solution $(M_1,M_2)>\textbf{0}$. 

\begin{lemma}\label{lemma2}
There exist $M_1, M_2 > 0$ such that $(M_1,M_2)$ is a super-solution of \eqref{eq:main}.

\end{lemma}

\begin{proof}
We split into two cases as follows

\textbf{Case 1:} $H'(0)G'(0)>a_{\min} b_{\min}$. One can check
\begin{align*}
&-a(x)M_1+H(M_2)\leq -a_{\min}M_1 + H(M_2),\\
&-b(x)M_2+G(M_1)\leq -b_{\min}M_2 + G	(M_1).
\end{align*}
Let us consider the system
\begin{align}\label{system1}
\begin{split}
\left\{\begin{array}{lllll}
-a_{\min}M_1 + H(M_2)=0,\\
-b_{\min}M_2 + G	(M_1)=0.
\end{array}\right.
\end{split}
\end{align}
Define $h(z)=G\left(\dfrac{H(z)}{{a_{\min}}}\right)-b_{\min} z$. Thanks to 
\ref{cond:cond2} and \ref{cond:cond3}, one has $h(\overline{z})<0$, $h(0)=0$ and $h'(0)>0$. It then follows that there exists $K_2\in (0,\overline{z})$ such that $h(K_2)=0$ and $h(K)<0$ for any $K>K_2$. 

Next, define $K_1 :=H(K_2)/a_{\min}$, one has $(K_1,K_2)$ is the unique positive solution of \eqref{system1}. On the other hand, for any $M_2>K_2$, setting $M_1 := H(M_2)/a_{\min}$ yields that $(M_1,M_2)$ satisfies
\begin{align}\label{system3}
\begin{split}
\left\{\begin{array}{lllll}
-a_{\min}M_1 + H(M_2)=0,\\
-b_{\min}M_2 + G	(M_1)<0.
\end{array}\right.
\end{split}
\end{align}
This implies that $(M_1,M_2)$ is a super-solution of \eqref{eq:main}.


\textbf{Case 2:} $H'(0)G'(0)\leq a_{\min} b_{\min}$. It is easy to check that
\begin{align*}
&-a(x)M_1+H(M_2)\leq -a_{\min}M_1 + H'(0)M_2,\\
&-b(x)M_2+G(M_1)\leq -b_{\min}M_2 + G'	(0)M_1.
\end{align*}
 Based on \cite[Lemma 3.2]{hsu_existence_2013}, there exists $(M_1,M_2)>\textbf{0}$ such that
\begin{align}\label{system2}
\begin{split}
\left\{\begin{array}{lllll}
-a_{\min}M_1 + H'(0)M_2<0,\\
-b_{\min}M_2 + G	'(0)M_1< 0.
\end{array}\right.
\end{split}
\end{align}
Consequently, $(M_1,M_2)$ is a super-solution of \eqref{eq:main} and so is $\alpha(M_1,M_2)$ for any constant $\alpha>0$.
\end{proof}

Next, we show the global-in-time existence.

\begin{proof}[\bf Proof of the first part of  Theorem \ref{theo:existence}.]
We split into two cases as follows

\textbf{Case 1: $H'(0)G'(0)>a_{\min}b_{\min}$}. Thanks to \ref{cond:cond4}, we choose $M_2>K_2$ large enough so that
\begin{align*}
\|u_0\|_{C(\overline{\O})}\leq \dfrac{H(M_2)}{a_{\min}}=:M_1,~\|v_0\|_{C(\overline{\O})}\leq M_2.
\end{align*}
One can prove that $(M_1,M_2)>\textbf{0}$ is a super-solution of \eqref{eq:main}.

\textbf{Case 2: $H'(0)G'(0)\leq a_{\min}b_{\min}$}. We choose $M_1,M_2>0$ large enough satisfying \eqref{system2}  so that
\begin{align*}
\|u_0\|_{C(\overline{\O})}\leq M_1,~\|v_0\|_{C(\overline{\O})}\leq M_2.
\end{align*} 
This follows from the fact that multiplying both sides of \eqref{system2} by a sufficiently large $\alpha>0$ yields the desired inequality, which shows that $(M_1,M_2)>\mathbf{0}$ is a super-solution of \eqref{eq:main}.

By applying the comparison principle in both cases, one has
\begin{align*}
0< u(t;u_0,v_0)\leq M_1,~0< v(t;u_0,v_0)\leq M_2,
\end{align*}
for any possible $t>0$. Note that the choice of $M_1$ and $M_2$ is independent of time. This ensures the global-in-time existence of the solution to \eqref{eq:main}.

\end{proof}


\section{\bf Steady state}\label{section5}
\subsection{\bf Existence and non-existence}

In this section, we investigate the existence of steady states of system \eqref{eq:main1} via the corresponding eigenvalue problem.
\begin{align}\label{eq:eigenvalue1}
\begin{split}
\left\{\begin{array}{lllll}
(-d_1\Delta_N)^{s_1}u+a(x)u-H'(0)v=\lambda u ,&x\in \O, \\
(-d_2\Delta_N)^{s_2}v+b(x)v-G'(0)u=\lambda v,&x\in \O.
\end{array}\right.
\end{split}
\end{align}
Based on this system, we choose
\begin{align*}
\textbf{A}:=\begin{pmatrix}
a(x)&-H'(0)\\
-G'(0)&b(x)
\end{pmatrix},~{\bf A_0}=-\textbf{A}.
\end{align*}
The existence of the principal eigenpair $(\lambda_p,\varphi_1)$ of \eqref{eq:eigenvalue1} on $C(\overline{\O})$ is ensured thanks to Theorem \ref{eigen_theo}.
%

Note that if $\mathbf{u}$ is a solution of \eqref{eq:main1}, then, by applying bootstrap arguments similar to those in Section~\ref{Section 2.2}, there exists $3/2 < s < 2$ with $s-[s] > \max\{s_1, s_2\}$ such that
\begin{align*}
u\in C(\overline{\O})\cap H^{s}(\O),~v\in C(\overline{\O})\cap H^{s}(\O),
\end{align*}
and $u,v>0$ on $\overline{\O}$. In addition, one has $u\in H^{s_1}(\O),v\in H^{s_2}(\O)$.

Next, we study the {\it basic reproduction number} $\mathcal{R}_0$, a threshold that determines the existence of a steady state. Recall the system \eqref{eq:linear}.
\begin{align*}
\begin{split}
\left\{\begin{array}{lllll}
u_t+(-d_1\Delta_N)^{s_1}u=-a(x)u+H'(0)v,&t>0,~x\in \O, \\
v_t+(-d_2\Delta_N)^{s_2}v=-b(x)v,&t>0,~x\in \O.
\end{array}\right.
\end{split}
\end{align*}

 We define $\mathcal{B}:=-(-\mathbf{d}\Delta_N)^{\mathbf{s}}\mathbf{u}+\mathcal{\bf D}$, where, recall,
\begin{align*}
\mathcal{\bf D}=\begin{pmatrix}
-a(x)&H'(0)\\
0&-b(x)
\end{pmatrix}.
\end{align*}
Based on Theorem \ref{eigen_theo}, the operator $\mathcal{B}$ admits a unique principal eigenvalue $\lambda_{\mathcal{B}}$ satisfying
\begin{align*}
\lambda_{\mathcal{B}}=s(\mathcal{B}):=\{Re(\lambda)\in \R:\lambda \in \sigma({\mathcal{B}})\},
\end{align*}
due to the fact that $\mathcal{B}$ has compact positive resolvents. Furthermore, with $(u,v)>{\bf 0}$ is eigenfunction associated with $\lambda_{\mathcal{B}}$, one has
\begin{align*}
-(-d_2\Delta_N)^{s_2}v-b(x)v=\lambda_{\mathcal{B}}v
\end{align*}
Multiplying the equation by $v$ and integrating over $\Omega$ yields a negative left-hand side, which implies that $\lambda_{\mathcal{B}} < 0$.

Next, we establish the relations between the basic reproduction number $\mathcal{R}_0$ and the principal eigenvalue $\lambda_p$ associated with \eqref{eq:eigenvalue1}.

\begin{theorem}\label{basic}
Assume  \ref{cond:cond1} to \ref{cond:cond4} hold. Then, the following statements hold.
\begin{enumerate}[label=(\roman*)]
\item \label{casei1} $\text{sign}(\mathcal{R}_0-1)=\text{sign}(s(\mathcal{B}+F))=-\text{sign}(\lambda_p)$,
where $\text{sign}(\cdot)$ is the sign-function and $\lambda_p$ is the principal eigenvalue associated with \eqref{eq:eigenvalue1}.

\item \label{caseii1} If $\mathcal{R}_0>0$, then $\mu = \mathcal{R}_0$ is the unique solution of $s\left(\mathcal{B}+\dfrac{1}{\mu}F\right)=0$.
\end{enumerate}

\end{theorem}

\begin{proof}

Following the results in \cite[Theorem 3.12]{thieme_spectral_2009}, we get
\begin{align*}
(\lambda-\mathcal{B})^{-1}{\bf u}_0=\int_0^\infty e^{-\lambda t} T(t){\bf u}_0dt,~\forall \lambda>s(\mathcal{B}),~{\bf u}_0\in \mathcal{C},~{\bf u}_0\geq \textbf{0}.
\end{align*}

Since $s(\mathcal{B})<0$, one can choose $\lambda =0$ so that
\begin{align*}
-\mathcal{B}^{-1}{\bf u}_0=\int_0^\infty  T(t){\bf u}_0dt \text{ implies that }-F\mathcal{B}^{-1}=L.
\end{align*}
Based on \cite[Theorem 3.5]{thieme_spectral_2009}, $\text{sign}(r(-{\bf F}\mathcal{B}^{-1})-1)=\text{sign}(s(\mathcal{B}+{\bf F}))=-\text{sign}(\lambda_p)$.

Next, if $s(-{\bf F}\mathcal{B}^{-1})=\mathcal{R}_0>0$, it follows from \cite[Theorem 5.2]{feng_principal_2024} that $\mu=\mathcal{R}_0$ is the unique solution of $s\left(\mathcal{B}+\dfrac{1}{\mu}{\bf F}\right)=0$. This concludes the proof.

\end{proof}


Now, we prove the comparison principle for the steady state system \eqref{eq:main1} by means of the classical sliding method.  

\begin{definition}
Let $\textbf{w}=(w_1,w_2)\in \mathcal{C}\cap \mathcal{H}^{\bf s}$.  
We call $\textbf{w}$ a \textit{super-solution} of \eqref{eq:main1} if
\begin{align*}
\begin{cases}
(-d_1\Delta_N)^{s_1} u \;\geq\; -a(x)u + H(v), & x\in \Omega, \\[0.3em]
(-d_2\Delta_N)^{s_2} v \;\geq\; -b(x)v + G(u), & x\in \Omega.
\end{cases}
\end{align*}
A \textit{sub-solution} is defined analogously by reversing the inequalities.
\end{definition}

\begin{theorem}[Comparison principle for steady state]\label{compar:1state}
Assume \ref{cond:cond1} to \ref{cond:cond4} hold. Consider ${\bf w}=(w_1,w_2)>0$ on $\overline{\O}$ is a super-solution of \eqref{eq:main1} and ${\bf v}=(v_1,v_2)>0$ on $\overline{\O}$ is a sub-solution. Then, one has
\begin{align*}
{\bf v}\leq {\bf w} \text{ on }\overline{\O}.
\end{align*}
\end{theorem}
\begin{proof}

We consider the following value
\begin{align*}
k^*=\inf\{k: \textbf{v}\leq k \textbf{w} \text{ on } \overline{\O}\}.
\end{align*}
It is clear that $k^*$ is well-defined and there exists $x_0\in \overline{\O}$ such that $\textbf{v}(x_0)=k^*\textbf{w}(x_0)$. We claim that $k^*\leq 1$. Assume by contradiction that $k^*>1$. 
Define $\textbf{z}=k^*\textbf{w}-\textbf{v}=(z_1,z_2)\geq (0,0)$. Then, $\textbf{z}(x_0)=\textbf{0}$ and
\begin{align*}
(-d_1\Delta_N)^{s_1}z_1+a(x)z_1\geq k^*H(w_2)-H(v_2)>H(k^*w_2)-H(v_2)>H'(k^*w_2)z_2\geq 0.
\end{align*}
This leads to
\begin{align*}
(-d_1\Delta_N)^{s_1}z_1+Cz_1>0
\end{align*}
In a similar manner, one has
\begin{align*}
(-d_2\Delta_N)^{s_2}z_2+Cz_2>0.
\end{align*}
for some $C>0$ sufficiently large, depending on bounds of $a,b$. Let us define
\begin{align*}
f_1:=((-d_1\Delta_N)^{s_1}+CI)z_1>0;~f_2:=((-d_2\Delta_N)^{s_2}+CI)z_2>0.
\end{align*}
One can check that 
\begin{align*}
z_1=((-d_1\Delta_N)^{s_1}+C I)^{-1}f_1;~z_2=((-d_2\Delta_N)^{s_2}+CI)^{-1}f_2.
\end{align*}
In view of \cite[Theorem 3.1, Step 3, Case 1]{zhao_spatiotemporal_2025}, one has
\begin{align*}
\text{  $z_1>0$ and $z_2>0$ on $\overline{\O}$, }
\end{align*}
which contradicts $\mathbf{z}(x_0)=0$. As the results, we obtain $\mathbf{v} \leq \mathbf{w}$ and complete the proof.
\end{proof}

We are now in a position to study the steady state solutions.



\begin{proof}[\bf Proof of Theorem \ref{theo:steady}]



\ref{Exis}. To prove the existence of the steady state the key point is to use the Schauder fixed point theorem. Since $\mathcal{R}_0>1$, one has $\lambda_p<0$.

Consider the following function 
\begin{align*}
\mathcal{F}(\textbf{u}):=((-\textbf{d}\Delta_N)^{\textbf{s}}+I)^{-1}[I+ \textbf{B}+\mathcal{G}] \text{ on }  \mathcal{X}.
\end{align*}
 Let $\{\textbf{u}_n=(u_n,v_n)\}\in \mathcal{X}$ be a bounded sequence. Since $H$ and $G$ are strictly concave, one has
\begin{align*}
H(v_1)\leq H(v_2)+H'(v_2)(v_1-v_2).
\end{align*}
Choosing $v_2=0$ and $v_1 =v_n$, we obtain that 
\begin{align*}
H(v_n)\leq H'(0)v_n.
\end{align*}
An analogous inequality holds for $G$. This implies that $\mathcal{G}\textbf{u}_n$ is a bounded sequence in $\mathcal{X}$. Then, the compactness of $((-\textbf{d}\Delta_N)^{\textbf{s}}+I)^{-1}$ implies that $\mathcal{F}$ is compact on $\mathcal{X}$.

Next, for $\epsilon>0$, define $\underline{\textbf{u}}=\epsilon \varphi_1$, where $\varphi_1>0$ is the eigenfunction associated with the principal  eigenvalue $\lambda_p$ of \eqref{eq:eigenvalue1}. It is known that $z\mapsto H(z)/z$ is non-increasing on $(0,\infty)$ and $\lim\limits_{z\rightarrow 0^+} H(z)/z=H'(0)$. Hence, for $\epsilon'=-\lambda_p \min\limits_{\overline{\O}}\dfrac{\varphi_{1,1}}{\varphi_{1,2}}>0$, there exists $\epsilon>0$ small enough such that
\begin{align*}
\dfrac{H(\epsilon\varphi_{1,2})}{\epsilon\varphi_{1,2}}-H'(0)\geq \dfrac{H\left(\epsilon\max\limits_{\O}\varphi_{1,2}\right)}{\epsilon\max\limits_{\O}\varphi_{1,2}}-H'(0)\geq -\epsilon'.
\end{align*}
As the results, we get that
\begin{align*}
H(\epsilon\varphi_{1,2}) >\lambda_p \epsilon\varphi_{1,1} + H'(0)\epsilon\varphi_{1,2}.
\end{align*}
Similarly,
\begin{align*}
G(\epsilon\varphi_{1,1}) >\lambda_p \epsilon\varphi_{1,2} + G'(0)\epsilon\varphi_{1,1}.
\end{align*}
Hence, $\underline{\textbf{u}}$ is a sub-solution of \eqref{eq:main1} with $\epsilon>0$ small enough. 

Next, let us define
\begin{align*}
\mathfrak{Y}:=\{\textbf{w}\in\mathcal{X}: \underline{\textbf{u}}\leq\textbf{w} \leq \overline{\textbf{u}} \},
\end{align*}
where $\overline{\textbf{u}} :=(M_1,M_2)$, $M_1,M_2>0$ are in Lemma \ref{lemma2}. It is easy to check that $\overline{\textbf{u}}$ is a super-solution of \eqref{eq:main1} and $\mathfrak{Y}$ is a convex set. Here, we can choose $\epsilon>0$ smaller so that $\underline{\textbf{u}}<\overline{\textbf{u}}$. By applying Schauder fixed point theorem for $\mathcal{F}$, there exists $\textbf{u}=(u,v)\in \mathfrak{Y}$ such that
\begin{align*}
\textbf{u}=\mathcal{F}(\textbf{u}),
\end{align*}
which is a solution of \eqref{eq:main1}. 

We now establish uniqueness. Let $\mathbf{v}=(v_1,v_2),~\mathbf{w}=(w_1,w_2) \in \mathcal{C} \cap \mathcal{H}^{\bf s}(\O)$ be two non-negative, nontrivial solutions of \eqref{eq:main1}. For $\beta>0$ large enough, define ${\bf f}:=[\beta I+ \textbf{B}+\mathcal{G}]\textbf{v}$, one can verify that
\begin{align*}
\textbf{v}=((-\textbf{d}\Delta_N)^{\textbf{s}}+\beta I)^{-1}[\beta I+ \textbf{B}+\mathcal{G}]\textbf{v}=((-\textbf{d}\Delta_N)^{\textbf{s}}+\beta I)^{-1}{\bf f};~{\bf f}\succeq {\bf 0}.
\end{align*}
It then follows from \cite[Theorem 3.1, Step 3, Case 1]{zhao_spatiotemporal_2025} that $\textbf{v}>{\bf 0}$. Similarly, $\textbf{w}>{\bf 0}$. By applying Theorem \ref{compar:1state}, one has $\mathbf{v} \leq \mathbf{w}$. Similarly, $\mathbf{w} \leq \mathbf{v}$ and, hence, $\mathbf{v} = \mathbf{w}$.


\ref{Non}. Since $\mathcal{R}_0\leq 1$, one has $\lambda_p\geq 0$.  For any $\epsilon>0$, it is known that
\begin{align*}
\begin{split}
\left\{\begin{array}{lllll}
(-d_1\Delta_N)^{s_1}\epsilon\varphi_{1,1}=\lambda_p \epsilon\varphi_{1,1}-a(x)\epsilon\varphi_{1,1}+H'(0)\epsilon\varphi_{1,2}\geq-a(x)\epsilon\varphi_{1,1}+H(\epsilon\varphi_{1,2})  , \\
(-d_2\Delta_N)^{s_2}\epsilon\varphi_{1,2}=\lambda_p \epsilon\varphi_{1,2}-b(x)\epsilon\varphi_{1,2}+G'(0)\epsilon\varphi_{1,1}\geq -b(x)\epsilon\varphi_{1,2}+G(\epsilon\varphi_{1,1}),
\end{array}\right.
\end{split}
\end{align*}
which implies that $\epsilon\varphi_1>0$ is a positive super-solution of \eqref{eq:main1}.

We prove by contradiction. Assume that there exists a non-trivial $\textbf{u}\geq \textbf{0}$ is a solution of \eqref{eq:main1}. Then, it follows from Theorem \ref{compar:1state} that
\begin{align*}
0\leq \textbf{u}\leq \epsilon \varphi_1,~\forall \epsilon>0.
\end{align*}
Letting $\epsilon \rightarrow 0$, we obtain that $\textbf{u} \equiv 0$, which is a contradiction.

 This concludes the proof.

%



\end{proof}


\subsection{\bf Stability of the steady state}

In this section, we investigate the stability of the steady state.

\begin{proof}[\bf Proof of second part of Theorem \ref{theo:existence}] Let us investigate the two separate cases.

\ref{casei}. The proof is divided by two steps.

\textbf{Step 1:} Assume that $t\mapsto \textbf{u}(t;\textbf{u}_0)$ is monotone. Define $\textbf{w}=(w_1,w_2):=\lim\limits_{t\rightarrow \infty}\textbf{u}(t;\textbf{u}_0)$ point-wise limit. Suppose further that $\textbf{w}$ is non-trivial. We will show that $\lim\limits_{t\rightarrow \infty}\textbf{u}(t;\textbf{u}_0)=\textbf{u}_1$ uniformly on $\overline{\O}$. 


Since $\mathcal{C}\subset\mathcal{X}$, it follows from \eqref{continuous} that
\begin{align*}
\|\textbf{u}(t;\textbf{u}(s,\textbf{u}_0))-\textbf{u}(t;\textbf{w})\|_{\mathcal{X}}\leq C_0\|\textbf{u}(s,\textbf{u}_0)-\textbf{w}\|_{\mathcal{X}}e^{C_0Kt}\rightarrow 0 \text{ as } s\rightarrow \infty,~\forall t>0.
\end{align*}
Here, we use the Lebesgue dominant theorem to prove the last convergence. As the results, one can check that
\begin{align*}
\textbf{w}=\lim_{s\rightarrow \infty}\textbf{u}(t+s;\textbf{u}_0)=\lim_{s\rightarrow \infty}\textbf{u}(t;\textbf{u}(s,\textbf{u}_0))=\textbf{u}(t;\textbf{w}) \text{ on }\mathcal{X},~\forall t>0.
\end{align*}
This shows that $\textbf{w}$ is a solution of \eqref{eq:main1}, which further implies that $\textbf{w}=\textbf{u}_1$. By applying Dini's theorem, we get the desired results.



\textbf{Step 2:} Prove the main result.

Thanks to the first part Theorem \ref{theo:existence}, we know that
\begin{align*}
\textbf{0}<\textbf{u}(t;\textbf{u}_0)\leq (M_1,M_2) \text{ on }\overline{\O},~\forall t>0,
\end{align*}
where $(M_1,M_2)$ is a super-solution of \eqref{eq:main}. For $T_0>0$, there exists $\delta>0$ such that
\begin{align*}
(\delta,\delta)< \textbf{u}(T_0;\textbf{u}_0)\text{ on }\overline{\O}.
\end{align*}
Let $\varphi_1\in \mathcal{C}$ be the positive eigenfunction associated with $\lambda_p$. Since $\lambda_p<0$,  we can choose $\epsilon>0$ small enough so that
\begin{align*}
\epsilon \varphi_1\text{ is a sub-solution and }\epsilon \varphi_1<\textbf{u}(T_0;\textbf{u}_0).
\end{align*}
By the comparison principle in Section \ref{section:3} and the uniqueness, one has
\begin{align*}
&t\mapsto \textbf{u}(t;\epsilon \varphi_1) \text{ is non-decreasing},\\
&t\mapsto \textbf{u}(t; (M_1,M_2)) \text{ is non-increasing},\\
&\textbf{u}(t;\epsilon \varphi_1)\leq \textbf{u}(t+T_0;\textbf{u}_0)\leq \textbf{u}(t;(M_1,M_2)),~\forall t>0.
\end{align*}
Passing to limit $t\rightarrow \infty$ and applying \textbf{Step 1}, we obtain the desired result.

\ref{caseii}. Choose $C>0$ large enough so that
\begin{align*}
\textbf{u}_0\leq C\varphi_1.
\end{align*}
Define $\overline{\bf v}(t,x)=(v_1,v_2)=C e^{-\lambda_p t}\varphi_1(x)$ for $t>0$ and $x\in \overline{\O}$. From \eqref{eq:eigenvalue1}, one can check that
\begin{align*}
\begin{split}
\left\{\begin{array}{lllll}
(v_1)_t+(-d_1\Delta_N)^{s_1}v_1=-a(x)v_1+H'(0)v_2, \\
(v_2)_t+(-d_2\Delta_N)^{s_2}v_2=-b(x)v_2+G'(0)v_1.
\end{array}\right.
\end{split}
\end{align*}
Because $H$ and $G$ are concave, $\overline{\mathbf{v}}$ serves as a super-solution of \eqref{eq:main1}. Furthermore, as $\overline{\mathbf{v}}(0,x) \geq \mathbf{u}_0(x)$ for every $x \in \overline{\Omega}$, we deduce that
\begin{align*}
0\leq \textbf{u}(t;\textbf{u}_0)\leq C e^{-\lambda_p t}\varphi_1,~\forall t>0.
\end{align*}
Since $\lambda_p > 0$, the desired result follows.

\end{proof}

\section{\bf Statements}

\textbf{Conflict of interest statement:} The authors have no conflicts of interest to declare that are relevant
to the content of this article.

\textbf{Data availability statement:} Data sharing not applicable to this article as no datasets were generated
or analyzed during the current study.


\end{document}